\documentclass{daj}

\dajAUTHORdetails{%
  title = {Permutations Contained in Transitive Subgroups}, 
  author = {Sean Eberhard, Kevin Ford, and Dimitris Koukoulopoulos},
  plaintextauthor = {Sean Eberhard, Kevin Ford, Dimitris Koukoulopoulos},
  keywords = {transitive groups, primitive groups, {\L}uczak--Pyber theorem},
}   

\dajEDITORdetails{%
   year={2016},
   number={12},
   received={5 May 2016},   
   published={29 July 2016},  
   doi={10.19086/da.849},       
}   

\usepackage{amsthm}
\usepackage{amsmath}
\usepackage{amssymb}
\usepackage{amsfonts}
\usepackage{setspace}
\usepackage{mathrsfs}
\usepackage{enumitem}

\usepackage{tikz}
\usetikzlibrary{arrows,shapes}

\DeclareMathAlphabet{\curly}{U}{rsfs}{m}{n}  

\newtheorem{theorem}{Theorem}[section]
\newtheorem{lemma}[theorem]{Lemma}
\newtheorem{proposition}[theorem]{Proposition}
\newtheorem{corollary}[theorem]{Corollary}

\definecolor{red}{rgb}{1,0,0}
\definecolor{orange}{rgb}{0.7,0.3,0}
\definecolor{blue}{rgb}{.2,.6,.75}
\definecolor{green}{rgb}{.0,.6,.2}

\theoremstyle{definition}

\theoremstyle{remark}
\newtheorem{remark}{Remark}[section]

\newcommand{\be}{\begin{equation}}
\newcommand{\ee}{\end{equation}}

\renewcommand{\leq}{\leqslant}
\renewcommand{\geq}{\geqslant}

\def\R{\mathbb{R}}

\def\Z{\mathbb{Z}}
\def\E{\mathbb{E}}
\def\P{\mathbb{P}}

\def\N{\mathbb{N}}

\def\X{\mathbf{X}}

\def\a{\mathbf{a}}
\def\b{\mathbf{b}}
\def\c{\mathbf{c}}
\def\d{\mathbf{d}}
\def\k{\mathbf{k}}
\def\DD{\curly D}

\newcommand{\CF}{\mathcal{F}}
\newcommand{\CG}{\mathcal{G}}
\newcommand{\CH}{\mathcal{H}}

\newcommand{\CP}{\mathcal{P}}
\newcommand{\CQ}{\mathcal{Q}}

\newcommand{\CS}{\mathcal{S}}

\def\eps{\varepsilon}

\def\tsum{{\textstyle\sum}}

\newcommand{\x}{\mathbf{x}}
\def\G{\mathcal{G}}
\def\B{\mathcal{B}}
\def\Vol{\textup{Vol}}
\def\bxi{\boldsymbol\xi}

\newcommand{\bs}\boldsymbol{}

\newcommand{\eq}[2]{ \begin{equation}\label{#1}\begin{split} #2 \end{split} \end{equation} }
\newcommand{\al}[1]{\begin{align} #1 \end{align} }
\newcommand{\als}[1]{\begin{align*} #1 \end{align*} }

\newcommand{\nn}{\nonumber \\}

\newcommand{\dee}{\mathrm{d}}

\newcommand{\lam}{\lambda}

\makeatletter
\renewcommand{\pmod}[1]{\allowbreak\mkern7mu({\operator@font mod}\,\,#1)}
\makeatother

\newcommand{\ssum}[1]{\sum_{\substack{#1}}}  
\renewcommand{\(}{\left(}
\renewcommand{\)}{\right)}

\newcommand{\pfrac}[2]{\left(\frac{#1}{#2}\right)}  
\newcommand{\order}{\asymp}      
\newcommand{\fl}[1]{{\ensuremath{\left\lfloor {#1} \right\rfloor}}}

\renewcommand{\le}{\leqslant}
\renewcommand{\ge}{\geqslant}

\newcommand{\cS}{\mathcal{S}}  
\newcommand{\cA}{\mathcal{A}}
\newcommand{\sL}{\mathscr{L}}

\newcommand{\mtr}[1]{ \left( \begin{matrix} #1 \end{matrix} \right) }

\newenvironment{romenumerate}{\begin{enumerate}  

 }{\end{enumerate}}

\numberwithin{equation}{section}

\begin{document}

\begin{frontmatter}[classification=text]


\author[eber]{Sean Eberhard}
\author[ford]{Kevin Ford\thanks{Supported by National Science Foundation Grant DMS-1501982.}}
\author[kouk]{Dimitris Koukoulopoulos\thanks{Supported by the National Science and Engineering Research Council of Canada and by the Fonds de recherche du Qu\'ebec -- Nature et technologies.}}

\begin{abstract}
In the first paper in this series we estimated the probability that a random permutation $\pi\in\cS_n$ has a fixed set of a given size. In this paper, we elaborate on the same method to estimate the probability that $\pi$ has $m$ disjoint fixed sets of prescribed sizes $k_1,\dots,k_m$, where $k_1+\cdots+k_m=n$. We deduce an estimate for the proportion of permutations contained in a transitive subgroup other than $\cS_n$ or $\cA_n$. This theorem consists of two parts: an estimate for the proportion of permutations contained in an imprimitive transitive subgroup, and an estimate for the proportion of permutations contained in a primitive subgroup other than $\cS_n$ or $\cA_n$. 
\end{abstract}
\end{frontmatter}

%
%
%
%
%
%
%
%
%

\tableofcontents

\section{Introduction}

In the first paper~\cite{EFG1} in this series we showed that the proportion $i(n,k)$ of permutations $\pi\in\cS_n$ having some fixed set of size $k$ is of order $k^{-\delta} (1+\log k)^{-3/2}$ uniformly for $1\leq k \leq n/2$, where $\delta = 1 - 1/(\log 2) - (\log\log 2)/(\log 2)$. If $n$ is even, it follows that the proportion of $\pi\in\cS_n$ contained in a transitive subgroup other than $\cS_n$ or $\cA_n$ is at least $cn^{-\delta}(\log n)^{-3/2}$ for some constant $c>0$. In that paper we stated our belief that a matching upper bound holds, and that stronger upper bounds hold for odd $n$. The purpose of the present paper is to prove this. Specifically, we prove the following theorem.

Here and throughout the paper the notation $X\order Y$ means that $c_1 Y \leq X \leq c_2 Y$ for some constants $c_1,c_2>0$. We will also use $X\ll Y$ to mean $X \leq cY$ for some constant $c$, as well as standard $O(\cdot)$ and $o(\cdot)$ notation.


\begin{theorem}\label{trans}
Let $T(n)$ be the proportion of $\pi\in\cS_n$ contained in a transitive subgroup other than $\cS_n$ or $\cA_n$, and let $p$ be the smallest prime factor of $n$. Then
\[
  T(n) \order
  \begin{cases}
    n^{-\delta_2}(\log n)^{-3/2} & \text{if}~p=2,\\
    n^{-\delta_3}(\log n)^{-3/2} & \text{if}~p=3,\\
    n^{-1+1/(p-1)} &\text{if $5\le p\ll1$},\\
    n^{-1+o(1)} & \text{if $p\to\infty$},
  \end{cases}
\]
where 
\als{
\delta_m 
	&= \int_1^{(m-1)/\log m} (\log t) \dee t \\
	&= 1 - \frac{m-1}{\log m} + \frac{(m-1)\log(m-1)}{\log m} - \frac{(m-1)\log\log m}{\log m}.
}\end{theorem}

\noindent
We record here the first few values of the sequence $\delta_m$ for easy reference:
\[
  \delta_2 = 0.08607\dots,  
  \quad \delta_3 = 0.27017\dots ,
  \quad
  \delta_4 = 0.50655\dots, 
  \quad \delta_5 = 0.77733\dots.
\]


The theorem that $T(n)\to0$ as $n\to\infty$ is due to {\L}uczak and Pyber~\cite{lp93}, whose method can be used to prove $T(n) = O(n^{-c})$ for some small $c>0$. This theorem has been widely hailed in the literature and has seen several applications: see for example Cameron and Kantor~\cite{camkan} for an application to the group generated by the first two rows of a random Latin square, Babai and Hayes~\cite{babhay} for an application to generating the symmetric group with one random and one fixed generator, Diaconis, Fulman, and Guralnick~\cite{dfg08} for an application to counting derangements in arbitrary actions of the symmetric group, and Kowalski and Zywina~\cite{kowzyw} and Eberhard, Green, and Ford~\cite{EFG2} for applications to invariable generation. The rate of decay of $T(n)$ had remained somewhat of a mystery, however, and this question was emphasized by Cameron and Kantor as well as by Babai and Hayes. Theorem~\ref{trans} therefore fills a rather large gap in our understanding of the subgroup structure of the symmetric group.

Theorem~\ref{trans} is actually a composite of two theorems, one about imprimitive transitive subgroups and one about primitive subgroups. Recall that a subgroup $H\leq \cS_n$ is called \emph{imprimitive} if it preserves some nontrivial partition of $\{1,\dots,n\}$ into blocks. If $H$ is transitive, then the blocks of such a partition must all have the same size. Therefore, if $I(n)$ is the proportion of $\pi\in\cS_n$ contained in an imprimitive transitive subgroup, and $I(n,\nu)$ is the proportion of $\pi\in\cS_n$ preserving some partition of $\{1,\dots,n\}$ into $\nu$ blocks of size $n/\nu$, then
\[
  I(n) \leq \ssum{\nu\mid n \\ 1<\nu<n} I(n,\nu).
\]
On the other hand, if $H$ does not preserve a nontrivial partition of $\{1,\dots,n\}$, then $H$ is called \emph{primitive}. Let $P(n)$ be the proportion of $\pi\in\cS_n$ contained in a primitive subgroup other than $\cS_n$ or $\cA_n$. We prove the following estimates for $I(n)$ and $P(n)$.


\begin{theorem}\label{imprimitive}
Let $\nu$ be a divisor of $n$. Then
\[
	I(n,\nu)\order
	\begin{cases}
	  n^{-\delta_\nu}(\log n)^{-3/2} & \text{if}~1<\nu\leq 4,\\
	  n^{-1+1/(\nu-1)} & \text{if}~5\leq \nu \leq \log n,\\
	  n^{-1} & \text{if}~\log n \leq \nu \leq n/\log n,\\
	  n^{-1+\nu/n} & \text{if}~n/\log n \leq \nu<n .
	\end{cases}
\]
Thus, if $n$ is composite and $p$ is the smallest prime factor of $n$, 
then $I(n)\order I(n,p) + n^{-1+O(1/\log\log n)}$, with $I(n,p)$ as above. 
\end{theorem}

\begin{remark}
The term $n^{-1+O(1/\log\log n)}$ cannot be completely removed. In Remark \ref{lb-rem}, we construct integers $n$ for which 
\[
 I(n) \gg   \frac{\log n}{\log\log n} I(n,p).
\]
\end{remark}


\begin{theorem}\label{primitive}
$P(n) \leq n^{-1+o(1)}$.
\end{theorem}


The theorem that $I(n)\to0$ as $n\to\infty$ is due to {\L}uczak and Pyber~\cite{lp93}. The somewhat older theorem that $P(n)\to 0$ as $n\to\infty$ is due to Bovey~\cite{bovey}, who proved the bound $P(n)\leq n^{-1/2+o(1)}$. More recently Bovey's estimate was improved to $P(n)\leq n^{-2/3+o(1)}$ by Diaconis, Fulman, and Guralnick~\cite[Section~7]{dfg08}, who also conjectured that $P(n)\leq O(n^{-1})$. In truth, $P(n)$ depends rather delicately on the arithmetic of $n$, and in fact $P(n)=0$ for almost all $n$ (see Cameron, Neumann, and Teague~\cite{cameron-neumann-teague}), but $O(n^{-1})$ would be the best possible bound which depends only on the size of $n$. For example if $n$ happens to be prime then every $n$-cycle generates a primitive subgroup; similarly, if $p = n-1$ is prime then every $n$-cycle is contained in a primitive subgroup isomorphic to $\textup{SL}_2(p)$. Our proof of the bound $n^{-1+o(1)}$ is essentially that of~\cite{dfg08}, except that we insert our new bound for $I(n,\nu)$ at a critical stage in the proof.

The proof of Theorem~\ref{imprimitive} is self-contained, except for a theorem we borrow from \cite{dfg08} to deal with $\nu$ of size $n^{1-o(1)}$. The proof of Theorem~\ref{primitive} on the other hand makes essential use of the classification of finite simple groups via work of Liebeck and Saxl~\cite{liebecksaxl} classifying primitive subgroups of small minimal degree (extended by Guralnick and Magaard~\cite{guralnickmagaard}).

The connection between $I(n,\nu)$ and $i(n,k)$ is easy to explain. Suppose $\pi$ preserves a partition of $\{1,\dots,n\}$ into $\nu$ blocks of size $n/\nu$. Then $\pi$ induces a permutation $\tilde\pi\in S_\nu$ on the set of blocks. If $\tilde\pi$ has cycle lengths $d_1,\dots,d_m$, then it follows that $\pi$ has disjoint fixed sets $A_1,\dots,A_m$ such that $|A_i| = d_in/\nu$ and such that all cycles of $\pi|_{A_i}$ are divisible by $d_i$. For example, assume that we have the permutation
\[
\pi = \mtr{
 1 & 2 & 3 &  4 & 5 & 6 & 7 & 8 & 9 \\
 4 & 5 & 6 &  1 & 2 & 3 & 8 & 9 & 7 
},
\]
counted by $I(9,3)$, since it permutes the blocks $\{1,2,3\}$, $\{4,5,6\}$ and $\{7,8,9\}$. Then the induced permutation $\tilde{\pi}$ is the permutation $\mtr{
 1 & 2 & 3 \\ 2 & 1 & 3},$ whose cycle lengths are 2 and 1. We may then take $A_1=\{1,2,3,4,5,6\}$ and $A_2=\{7,8,9\}$, which are both fixed subsets of $\pi$. In addition, $\pi|_{A_1}=(1\,4)(2\,5)(3\,6)$ consists only of 2-divisible cycles.

The converse to the above relation holds as well : if $\pi$ has disjoint fixed sets $A_1,\dots,A_m$ such that $|A_i| = d_in/\nu$ and such that all cycles of $\pi|_{A_i}$ are divisible by $d_i$, then $\pi$ preserves a system of $\nu$ blocks of size $n/\nu$. We are thus naturally led to the following definition: for $\k=(k_1,\dots,k_m)$ such that $\sum_{i=1}^m k_i = n$ and $\d=(d_1,\dots,d_m)$, let $i(n,\k,\d)$ be the proportion of $\pi\in\cS_n$ having disjoint fixed sets $A_1,\dots,A_m$ such that $|A_i|=k_i$ and such that all cycles of $\pi|_{A_i}$ are divisible by $d_i$. Then we have
\eq{Innubound}{
  \max_{d_i} i(n,(d_i n/\nu)_i,(d_i)_i) \leq I(n,\nu) \leq \sum_{d_i} i(n,(d_in/\nu)_i,(d_i)_i),
}
where the max and sum run over partitions $(d_1,\dots,d_m)$ of $\nu$. Thus, at least for small $\nu$, it suffices to understand $i(n,\k,\d)$.

Moreover, it turns out that the only nontrivial case for which we need sharp bounds is the case in which $d_i=1$ for each $i$. In this case we write just $i(n,\k)$ for $i(n,\k,\d)$: this is simpy the proportion of permutations $\pi$ having disjoint fixed sets of sizes $k_1,\dots,k_m$. Our main task therefore is to establish the following estimate for $i(n,\k)$. Note that because $i(n,k) = i(n,(k,n-k))$, this generalizes the main result of~\cite{EFG1}.

\begin{theorem}\label{main}
Let $m\geq 2$ and assume $2\leq k_1\leq \cdots \leq k_m$ and $\sum_{i=1}^m k_i = n$. Then
\[
  i(n,\k) \ll_m (k_{m-1}/k_1)^m k_1^{-\delta_m} (\log k_1)^{-3/2}  .
\]
Moreover, if $k_{m-1} \leq c k_1$ then
\[
  i(n,\k) \order_{m,c} k_1^{-\delta_m} (\log k_1)^{-3/2}.
\]
In particular, if $k_i \order_m n$ for each $i$ then
\[
  i(n,\k) \order_m n^{-\delta_m} (\log n)^{-3/2}.
\]
\end{theorem}

In~\cite{EFG1}, we relied on an analogy with analytic number theory wherein the problem of estimating $i(n,k)$ corresponds to the problem of estimating the proportion of integers $n\leq x$ with a divisor in a given dyadic interval $(y,2y]$: this is the so-called multiplication table problem, which was solved up to a constant factor by the second author~\cite{ford,ford-2}. Similarly, the problem of estimating $i(n,\k,\d)$ is related to higher-dimensional versions of the multiplication table problem. The connection is closest for $i(n,\k)$, which under the analogy corresponds to the proportion of $n\leq x$ that are decomposable as $n_1\cdots n_m$ with $n_i\in(y_i,2y_i]$ for each $i$. Except in some cases in which the sizes of the parameters $y_i$ are too wildly different, this proportion was computed up to a constant factor by the third author~\cite{kouk1,kouk2}. For comparison with Theorem~\ref{main}, refer in particular to~\cite[Theorem~1]{kouk1}. Thus, as in~\cite{EFG1}, the task of proving of Theorem~\ref{main} is largely one of translation.

Given the strength of the analogy with \cite[Theorem~1]{kouk1}, one might hope to be able to deduce the result directly using transference ideas. While unfortunately this does not appear to be possible, the basic outline of the proof is the same.

When the vector $\d$ is allowed to be arbitrary, however, there are some additional complications, and while there is still some connection with the generalized multiplication table problem, in fact it is somewhat fortunate that the partitions of $\nu$ constituting the main contribution to $I(n,\nu)$ correspond to $\d$ for which we know how to estimate $i(n,\k,\d)$ satisfactorily, while for the rest we can get away with a crude bound.


We have made an effort to follow the exposition and technical notation previously used in~\cite{ford,ford-2,kouk1,kouk2,EFG1}, but unfortunately many notational clashes have been unavoidable.


\medskip

{\bf Acknowledgments.} We would like to thank Ben Green for helpful conversations.


\section{Outline of the proof}\label{sec:outline}

In this section we sketch the broad idea and initial reductions involved in the proof of Theorem~\ref{imprimitive}. The proof of Theorem~\ref{primitive} relies on Theorem~\ref{imprimitive} but is otherwise unrelated, so we defer discussion to Section~\ref{sec:primitive}.

Let $\nu$ be a proper nontrivial divisor of $n$. When $\nu$ becomes large we will survive on a combination of crude arguments and previous work of Diaconis, Fulman, and Guralnick~\cite{dfg08}, so in this outline assume $\nu$ is bounded. As explained in the introduction, our starting point is the relation \eqref{Innubound}, whence we immediately infer that
\[
  I(n,\nu) \order_\nu \max_{d_i} i(n,(d_in/\nu)_i,(d_i)_i) .
\]
The estimation of $I(n,\nu)$ for $\nu$ bounded is thus immediately subsumed by the general problem of estimating $i(n,\k,\d)$.

Call a partition $(d_i)$ of $\nu$ maximizing $i(n,(d_in/\nu)_i,(d_i)_i)$ \emph{dominant}. There is a comparatively simple bound for $i(n,\k,\d)$ which already shows that, for every $\nu$, every dominant partition has the form $(d,1,\dots,1)$ for some $d\geq 1$.


\begin{lemma}\label{first-lemma}\ 
\begin{enumerate}[label={\upshape(\alph*)}]
 \item If $d\mid n$, then the proportion, $i(n,(n),(d))$ of $\pi\in\cS_n$ all of whose cycle lengths are divisible by $d$ satisfies $n^{-1+1/d} \ll i(n,(n),(d)) \le n^{-1+1/d}$.
 \item If $n=n'+n''$, $\k=(\k',\k'')$, and $\d=(\d',\d'')$, then
 \[
   i(n,\k,\d) \leq i(n',\k',\d') \, i(n'',\k'',\d'').
 \]
 Here, we assume of course that $\k'$ and $\d'$ have the same length $m'$, $\k''$ and $\d''$ have the same length $m''$, $\sum_i k'_i = n'$, and $\sum_i k''_i = n''$
 \item For every $\k$ and $\d$, we have that
 \[
   i(n,\k,\d) \leq k_1^{-1+1/d_1} \cdots k_m^{-1+1/d_m}.
 \]
 \item For every fixed $\nu\geq 1$ and sufficiently large $n$, 
 every dominant partition of $\nu$ has the form $(d,1,\dots,1)$ for some $d\geq 1$.
\end{enumerate}
\end{lemma}


\begin{proof} (a) The bound is trivial when $d=1$, so we may suppose that $d\ge 2$. This is a well-known result, which can be proved as follows: let $f_d(n)$ be the number of permutations $\pi\in\cS_n$ having all cycle lengths divisible by $d$. Then certainly $f_d(0)=1$, and for $n\geq d$ we claim
\eq{rdrecurrence}{
	f_d(n) = (n-1)\cdots(n-d+2)(n-d+1)^2f_d(n-d).
}
Indeed, to choose a permutation $\pi$ all of whose cycles are $d$-divisible, first choose $d-1$ distinct cyclic elements $\pi(1),\dots,\pi^{d-1}(1)$ from $\{2,\dots,n\}$, then choose $\pi^d(1)$ from $\{1,\dots,n\}\setminus\{\pi(1),\dots,\pi^{d-1}(1)\}$, then choose a permutation $\pi'$ of the $(n-d)$-element set $\{1,\dots,n\}\setminus\{1,\pi(1),\dots,\pi^{d-1}(1)\}$ all of whose cycles are $d$-divisible. If $\pi^d(1)=1$ then we let $\pi$ coincide with $\pi'$ on $\{1,\dots,n\}\setminus\{1,\dots,\pi^{d-1}(1)\}$; if $\pi^d(1)\neq 1$ then we let $\pi(x)=\pi'(x)$ for all $x\ne 1$, and we lastly define $\pi(\pi^d(1))=\pi'(1)$.
 There are $(n-1)\cdots(n-d+1)$ choices for $\pi(1),\dots,\pi^{d-1}(1)$, $n-d+1$ choices for $\pi^d(1)$, and $f_d(n-d)$ choices for $\pi'$, so this proves~\eqref{rdrecurrence}.

Now, if $d\mid n$, then from~\eqref{rdrecurrence} we have
\als{
	f_d(n) 
	&= (n-1)!\frac{n-d+1}{n-d}\frac{n-2d+1}{n-2d}\cdots\frac{d+1}{d}\\
	&= (n-1)!\prod_{j=1}^{n/d-1} \(1 + \frac1{jd}\)\\
	&= (n-1)!\exp\(\sum_{j=1}^{n/d-1} \frac1{jd} + O\(\frac{1}{d^2} \)  \)    \\
	&= (n-1)!\exp\(\frac{\log(n/d)+\gamma}{d} + O\(\frac{1}{n}+\frac{1}{d^2}\)  \), 
}
where $\gamma$ is the Euler--Mascheroni constant. This proves the lower bound. When $d\ge 3$, we also have
\[
\prod_{j=1}^{n/d-1} \(1 + \frac1{jd}\) 
	\le \exp \(  \frac{1}{d}\sum_{j=1}^{n/d-1} \frac{1}{j}\) \\
	\le \exp \( \frac{1}{d} \(1 + \log \frac{n}{d} \) \)
	\le n^{1/d},
\]
proving the upper bound in this case.  When $d=2$, one checks by hand that the
inequality holds for $n<8$, and for $n\ge 8$ we have
\begin{align*}
\prod_{j=1}^{n/d-1} \(1 + \frac1{jd}\) 
	\le \frac32 \cdot \frac54 \exp \(\frac12 \sum_{j=3}^{n/2-1} \frac{1}{j} \)
	&\le \frac{15}{8} \exp \( \frac12 \int_2^{n/2} \frac{dt}{t} \)
		= \frac{15}{8} (n/4)^{1/2} 
		= \frac{15}{16}n^{1/2},
\end{align*}
proving the upper bound in this case as well.

\medskip

(b) We bound $n!\cdot i(n,\k,\d)$ by the sum, over all choices of a subset $A\subset\{1,\dots,n\}$ of size $n'$, of the number of ways of choosing a permutation $\pi|_A$ with disjoint fixed sets $A_1,\dots,A_{m'}$ and a permutation $\pi|_{A^c}$ with disjoint fixed sets $A_{m'+1},\dots,A_m$, both such that, for each $i$, $|A_i|=k_i$ and $\pi|_{A_i}$ has only $d_i$-divisible cycles. This proves that
\[
  n! \cdot i(n,\k,\d) \leq \binom{n}{n'} (n'! \cdot i(n',\k',\d')) ( n''! \cdot i(n'',\k'',\d'') ) ,
\]
which is equivalent to (b).

\medskip

(c) This follows immediately from parts (a) and (b). 

\medskip

(d) If $(d_i)$ is a partition of $\nu$ having at least two $d_i\geq 2$, then
\[
  \sum_{i=1}^m (1-1/d_i) \geq 1,
\]
so by part (c),
\[
  i(n,(d_in/\nu)_i,(d_i)_i) \le  (n/\nu) ^{-\sum_{i=1}^m (1-1/d_i)} \le (n/\nu)^{-1}.
\]
On the other hand, part (b) implies that
\[
  i(n,(n),(\nu)) \order n^{-1+1/\nu},
\]
so $(d_i)$ does not maximize $i(n,(d_in/\nu)_i,(d_i)_i)$ if $n$ is large enough.
\end{proof}


Though in general $i(n,\k,\d)$ is a rather subtle quantity, the case $\d=(1,1)$ for instance being the subject of the paper~\cite{EFG1}, some cases are elementary. For instance in Lemma~\ref{first-lemma}(a) we saw rather simply that $i(n,(n),(d)) \order n^{-1+1/d}$. It turns out that estimation of $i(n,(k_1,k_2),(d,1))$ is also elementary whenever $d\geq 3$.


\begin{lemma}\label{partition-d-1} Let $d\geq3$, and assume $k_1,k_2\geq 1$ and that $k_1$ is divisible by $d$. Then
\[
  i(n,(k_1,k_2),(d,1)) \order k_1^{-1+1/d}.
\]
\end{lemma}


\begin{proof}
The upper bound is contained in Lemma~\ref{first-lemma}(c). Recall the proof, which follows from parts (a) and (b) of that lemma: The number of ways of choosing a set $A_1$ of size $k_1$ is $\binom{n}{k_1}$, and the number of $\pi\in \cS_k$ having all cycles divisible by $d$ is $\order k!/k^{1-1/d}$, so
\[
  i(n,(k_1,k_2),(d,1)) \ll \frac1{n!} \binom{n}{k_1} \frac{k_1!}{k_1^{1-1/d}} k_2! = k_1^{-1+1/d}.
\]

Given $\pi\in\CS_n$, let $X=X(\pi)$ denote the number of acceptable choices for sets $A_1$ of size $k_1$ that are fixed by $\pi$ and such that $\pi|_{A_1}$ consists of $d$-divisible cycles. Then the argument in the above paragraph uses the simple relations
\[
  i(n,(k_1,k_2),(d,1)) = \P(X>0)\le \E X,
\]
where the underlying probability measure is the uniform measure on $\CS_n$, and then proceeds by showing that $\E X\asymp k_1^{-1+1/d}$. To find a matching lower bound, we will compute the second moment $\E X^2$, or in other words the number of pairs of $k_1$-sets $A_1,A_1'$ such that $\pi$ fixes both $A_1$ and $A_1'$ and such that $\pi|_{A_1}$ and $\pi|_{A_1'}$ are both wholly composed of $d$-divisible cycles. Note then that $\pi$ must fix each of the sets $A_1\cap A_1'$, $A_1\setminus A_1'$, $A_1'\setminus A_1$, and the restriction of $\pi$ to each of these sets must be wholly composed of $d$-divisible cycles. The number of ways of choosing two sets of size $k_1$ which overlap in a set of size $k_{11}$ is
\[
  \binom{n}{k_{11},k_1-k_{11},k_1-k_{11},k_2 - k_1 + k_{11}},
\]
so we deduce that
\als{
  \E X^2
  &\ll \ssum{0\le k_{11}\le k_1 \\ d|k_{11}} \frac1{n!} \binom{n}{k_{11},k_1-k_{11},k_1-k_{11},k_2 - k_1 + k_{11}} \frac{k_{11}!(k_1-k_{11})!^2(k_2 - k_1 + k_{11})!}{(k_{11}+1)^{1-1/d}(k_1-k_{11}+1)^{2-2/d}}\\
  &= \ssum{0\le k_{11}\le k_1 \\ d|k_{11}} \frac{1}{(k_{11}+1)^{1-1/d}(k_1-k_{11}+1)^{2-2/d}}   \\
  &\ll \ssum{0\le k_{11}\le k_1/2 \\ d|k_{11} } \frac1{(k_{11}+1)^{1-1/d}k_1^{2-2/d}}
  	+ \ssum{k_1/2<k_{11}\le k_1 \\ d|k_{11}} \frac{1}{k_1^{1-1/d}(1+k_1-k_{11})^{2-2/d}} \\
  &\order k_1^{-1+1/d}.
}
Hence by Cauchy--Schwarz we have
\[
  \P(X>0) \geq \frac{(\E X)^2}{\E X^2} \order k_1^{-1+1/d}.
\]
This proves the lemma.
\end{proof}


On the other hand, estimation of $i(n,\k)$ (that is, $i(n,\k,\d)$ in the case in which $d_i=1$ for each $i$) is not nearly so straightforward, and most of the paper will be devoted to establishing an estimate in this case, namely Theorem~\ref{main}. The proof of this theorem is divided over the next three sections. Specifically we prove a useful local-global principle in Section~\ref{sec:local-global}, we then prove the upper bound in Section~\ref{sec:upper}, and finally we prove the lower bound in Section~\ref{sec:lower}.

Assuming that we have proved Theorem~\ref{main}, we can then combine our various bounds for $i(n,\k,\d)$ to determine the dominant partition of $\nu$ for each bounded $\nu$. Moreover, since we have a sharp estimate for $i(n,(d_in/\nu)_i,(d_i)_i)$ for each such dominant partition, we are able to deduce a sharp estimate for $I(n,\nu)$.


\begin{proposition}\label{dominantpart}
Assume $\nu$ is bounded and $n$ is large. If $\nu\leq 4$, then the unique dominant partition of $\nu$ is $(1,\dots,1)$, while if $\nu\geq 5$, then the unique dominant partition of $\nu$ is $(\nu-1,1)$.
\end{proposition}


\begin{proof}
For $\nu=2$, it suffices to observe from Lemma~\ref{first-lemma}(a) that $i(n,(n),(2))\order n^{-1/2}$, while by Theorem~\ref{main} (alternatively, the main result of~\cite{EFG1}) we have $i(n,(n/2,n/2)) = n^{-\delta_2+o(1)}$. Since $\delta_2 = 0.08\ldots < 1/2$, the dominant partition of $2$ is $(1,1)$.

\medskip

Similarly, for $\nu=3$, just observe that $i(n,(n),(3))\order n^{-2/3}$ by Lemma~\ref{first-lemma}(a),
\[
  i(n,(2n/3,n/3),(2,1)) \ll n^{-1/2}
\]
by Lemma~\ref{first-lemma}(c), and
\[
  i(n,(n/3,n/3,n/3)) \order n^{-\delta_3 + o(1)}
\]
by Theorem~\ref{main}. Since $\delta_3 = 0.27\ldots < 1/2$, the dominant partition of $3$ is $(1,1,1)$.

\medskip

For $\nu=4$, again, observe that $i(n,(n),(4)) \order n^{-3/4}$, that
\[
  i(n,(3n/4,n/4),(3,1)) \ll n^{-2/3},
\]
and that
\[
  i(n,(n/4,n/4,n/4,n/4)) =    n^{-\delta_4+o(1)}.
\]
By Lemma~\ref{first-lemma}(d), the only other partition we need to consider is $(2,1,1)$, and for this partition we have from Lemma~\ref{first-lemma}(b) and Theorem~\ref{main} that
\begin{align*}
  i(n,(n/2,n/4,n/4),(2,1,1))
  &\le i(n/2,(n/2),(2)) \, i(n/2,(n/4,n/4))\\
  &= n^{-1/2-\delta_2 + o(1)}.
\end{align*}
Since $\delta_4 = 0.506\ldots$, while $1/2 + \delta_2 = 0.508\ldots$, the dominant partition of $4$ is $(1,1,1,1)$.

\medskip

Now, assume that $\nu\geq 5$. By Lemma~\ref{first-lemma}(d) we need only consider partitions of the form $(d,1,\dots,1)$. By parts (a) and (b) of Lemma~\ref{first-lemma}, combined with Theorem~\ref{main}, we have
\begin{align*}
  i(n,(dn/\nu,n/\nu,\dots,n/\nu),(d,1,\dots,1))
  &\leq i(dn/\nu, (dn/\nu), (d)) \, i((\nu-d)n/\nu,(n/\nu,\dots,n/\nu))\\
  &= n^{-1+1/d-\delta_{\nu-d} + o_\nu(1)}
\end{align*}
whenever $d\leq \nu - 2$. We use this bound only when $d\geq 2$, since when $d=1$ by Theorem~\ref{main} we have the slightly stronger bound
\[
  i(n,(n/\nu,\dots,n/\nu)) = n^{-\delta_\nu + o(1)}.
\]
Meanwhile, by Lemma~\ref{first-lemma}(a) we have
\[
  i(n,(n),(\nu)) \order n^{-1+1/\nu},
\]
which is always negligible since by Lemma~\ref{partition-d-1} we have
\[
  i(n,((\nu-1)n/\nu,n/\nu),(\nu-1,1)) \order n^{-1+1/(\nu-1)}.
\]
Thus the exponents we are comparing are
\[
\delta_\nu,\quad 1-\frac{1}{d} + \delta_{\nu - d} \quad (2\leq d\leq \nu-2),\quad 1-\frac{1}{\nu-1},
\]
and we claim that the last of these is the smallest whenever $\nu\geq 5$. 

Since $\delta_m=\int_1^{(m-1)/\log m}(\log t)dt$, the sequence $(\delta_m)_{m\ge2}$ is increasing. In particular, $\delta_\nu\ge\delta_6>1$ for $\nu\ge6$, and one checks by direct computation that $\delta_5=0.77\ldots>1-1/4$ too.

Next, if $2\le d\le \nu-4$, then
\[
1-\frac{1}{d} + \delta_{\nu - d} \ge \frac{1}{2}+\delta_4>1 .
\]
So, it remains to show that $1-1/d+\delta_{\nu-d}>1-1/(\nu-1)$ when $d\in\{\nu-3,\nu-2\}$. Writing $d=\nu-j$, this amounts to proving that
\[
\delta_j > \frac{j-1}{(\nu-1)(\nu-j)} \quad(j\in\{2,3\},\ \nu\ge5) 
\quad\Leftrightarrow\quad
\delta_j > \frac{j-1}{4(5-j)} \quad(j\in\{2,3\}) ,
\]
which one checks by direct computation.
\end{proof}


This completes the sketch of the proof of Theorem~\ref{imprimitive} when $\nu$ is bounded. As $\nu$ begins to grow with $n$, we must be more careful about some of our bounds, but we can afford to be more relaxed about others, and, by and large, the proof becomes simpler, using as key input Lemma \ref{first-lemma} and the case $m=2$ of Theorem~\ref{main}. As $\nu$ becomes very large, say of size $n^{1-o(1)}$, then our method begins to falter, and we outsource most of the work to~\cite{dfg08}. For all this, see Section~\ref{sec:imprimitive}.


\section{A local-to-global principle}\label{sec:local-global}

Given a $k$-tuple $\c=(c_1,\dots,c_k)$ of nonnegative integers, let $\sL_m(\c)$ be the set of all $m$-tuples
\[
  \(\sum_{j=1}^k jx_{1j}, \dots, \sum_{j=1}^k jx_{mj}\),
\]
where $(x_{ij})$ is an $m\times k$ matrix whose entries are nonnegative integers such that $\sum_{i=1}^m x_{ij} = c_j$ for each $j$. Note then that $i(n,\k)$ is precisely the probability of the event $\k\in\sL_m(\c)$, where $\c$ is the cycle type of a random permutation: here we say that $\pi\in\cS_n$ has \emph{cycle type} $\c$ if $\pi$ has exactly $c_j$ $j$-cycles for each $j\leq n$.
Instead of measuring this probability directly, however, we will use a convenient local-to-global principle which relates $i(n,\k)$ to the average size of $\sL_m(\c)$, given in Proposition \ref{local-global} below. The terminology `local-to-global' means that we turn a question about the local distribution of the set $\sL_m(\c)$ (whether it contains the point $\k$) to a question about its global distribution. Notice that if $k_{m-1}\ll k_1=k$, then a naive heuristic implies that the event $\k\in\sL_m(\c)$ occurs with probability $\approx |\sL_m(\c)|/k^{m-1}$. Our local-to-global estimate proves that this naive heuristic is true on average: 


\begin{proposition}\label{local-global}
Let $k = k_1$, and let $\X_k = (X_1,\dots,X_k)$, where $X_1,\dots,X_k$ are independent Poisson random variables with $\E X_j = 1/j$. Then
\[
  i(n,\k) \ll_m \(\frac{k_{m-1}}{k}\)^{m} \frac{\E|\sL_m(\X_k)|}{k^{m-1}}.
\]
Moreover if $k_{m-1} \leq c k_1$ then
\[
  i(n,\k) \order_{m,c} \frac{\E|\sL_m(\X_k)|}{k^{m-1}}.
\]
\end{proposition}


We start with a few basic upper bounds for $\sL_m(\c)$. Throughout this section we will denote by $P_{m-1}$ the projection onto the first $m-1$ coordinates, and we will often use the observation that $|\sL_m(\c)| = |P_{m-1}\sL_m(\c)|$: this holds simply because $\sL_m(\c)$ is contained in the hyperplane of $\R^m$ defined by $x_1+\cdots+x_m = \sum_{j=1}^k j c_j$.


\begin{lemma}\label{Lclem-higherdim} Let $\c=(c_1,\dots,c_k)$ and $\c'=(c'_1,\dots,c'_k)$.
\begin{enumerate}[label={\upshape(\alph*)}]
\item $|\sL_m(\c+\c')|\leq |\sL_m(\c)|\cdot|\sL_m(\c')|.$
\item $|\sL_m(\c)| \leq m^{c_1+\cdots+c_k}.$
\item If $c'_{j_1}=\cdots=c'_{j_h}=0$ and $c'_j=c_j$ for all other $j$, then ${|\sL_m(\c)| \leq  |\sL_m(\c')|m^{c_{j_1}+\cdots+c_{j_h}}.}$
\end{enumerate}
\end{lemma}


\begin{proof}
(a) Suppose $(x_{ij})$ is such that $\sum_i x_{ij} = c_j+c'_j$ for each $j$. We can find $(y_{ij})$ and $(z_{ij})$ such that $x_{ij} = y_{ij} + z_{ij}$ for all $i,j$, and such that $\sum_i y_{ij} = c_j$ and $\sum_i z_{ij} = c'_j$ for each $j$. Thus $\sL_m(\c+\c')\subset \sL_m(\c)+\sL_m(\c'),$ so (a) holds. 

\medskip

(b) We have that
\[
|\sL_m(\c)|\le \prod_{j=1}^k  |\{(x_{1j},\dots,x_{mj}): x_{1j}+\cdots +x_{mj}=c_j\}|
	= \prod_{j=1}^k  \binom{m+c_j-1}{c_j} \le m^{c_1+\cdots+c_k},
\] 
as claimed.

\medskip

(c) The claimed inequality follows immediately from parts (a) and (b).
\end{proof}


\begin{lemma}\label{lem555-higherdim}
Suppose that $k \leq k'$. Then $\E |\sL_m(\X_{k'})| \le (k'/k)^{m-1} \E |\sL_m(\X_k)|.$
\end{lemma}


\begin{proof}
By Lemma~\ref{Lclem-higherdim}(c), we have
\begin{align*}
\E|\sL_m(\X_{k'})|
&\leq \E \left[ |\sL_m(\X_k)| m^{X_{k+1}+\cdots+X_{k'}}   \right]\\
&=\E|\sL_m(\X_k)|\prod_{j=k+1}^{k'} \E\left[ m^{X_j} \right] \\
&=\E|\sL_m(\X_k)|\prod_{j=k+1}^{k'} e^{(m-1)/j} .
\end{align*}
Since $\sum_{j=k+1}^{k'} 1/j \le \int_k^{k'} dt/t = \log (k'/k)$, the claimed result follows.
\end{proof}


We need some further notation in connection with type vectors $\c = (c_1,\dots,c_k)$. We define
\[
  S(\c) = \sum_{j=1}^k j c_j.
\]
If $\c = (c_1,\dots,c_n)$ is the cycle type of some $\pi\in\cS_n$ then note that $S(\c)=n$. Occasionally however we will keep track of cycle types of partial permutations, in which case $S(\c)$ can be thought of as the total length represented by $\c$. We define also $C^+(\c)$ to be the largest $j$ such that $c_j>0$, or else zero if none exists. Similarly we define $C^-(\c)$ to be the smallest $j$ such that $c_j>0$, else $\infty$ if none exists. If $\c$ is the cycle type of $\pi\in\cS_n$ then $C^+(\c)$ and $C^-(\c)$ are the lengths of respectively the longest and shortest cycles of $\pi$; we will take the liberty of also using the alternative notation $C^+(\pi)$ and $C^-(\pi)$ to denote the same quantities.

\begin{lemma}\label{lem-sum-gen-higherdim}\ 
\begin{enumerate}[label={\upshape(\alph*)}]
\item Suppose $j_1,\dots, j_h \leq k$ are distinct integers and $a_1,\dots, a_h$ are positive integers. Then
\[
\E \left[|\sL_m(\X_k)| X_{j_1}^{a_1} \cdots X_{j_h}^{a_h} \right]
	\leq \frac{e^{m(2^{a_1}+\cdots+2^{a_h})}}{j_1 \dots j_h} \E |\sL_m(\X_k)| .
\]
\item For each fixed $r\ge1$, we have that 
\[
\E \left[ |\sL_m(\X_k)| S(\X_k)^r \right] 
	\ll_{r,m} k^r \E |\sL_m(\X_k)| .
\]
\item For each fixed $r\ge1$, we have that
\[
\E \left[ \frac{|\sL_m(\X_k)|}{\max\{C^+(\X_k),k-S(\X_k)\}^r} \right] 
	\ll_{r,m} \frac{\E |\sL_m(\X_k)|}{k^r} .
\]
\end{enumerate}
\end{lemma}


\begin{proof} (a) Define $\X'_k$ by putting $X'_{j_1}=\cdots=X'_{j_h}=0$ and $X'_j=X_j$ for all other $j$. By Lemma~\ref{Lclem-higherdim}(c), we have $ |\sL_m(\X_k)| \leq |\sL_m(\X'_k)| m^{X_{j_1}+\cdots+X_{j_h}}.$ Thus by independence
\[
\E \left[ |\sL_m(\X_k)|X_{j_1}^{a_1}\cdots X_{j_h}^{a_h} \right]
	\leq \E\left[ |\sL_m(\X'_k)|\right] 
		\prod_{i=1}^h \E \left[X_{j_i}^{a_i} m^{X_{j_i}}\right].
\]
The result follows immediately from this, the observation that $ \E|\sL_m(\X'_k)| \leq \E|\sL_m(\X_k)|$, and the bound
\[
	\E \left[X_j^a m^{X_j} \right]
		= e^{-1/j}\sum_{r=1}^\infty r^a \frac{(m/j)^r}{r!}
		 \leq	\sum_{r=1}^\infty 2^{ar} \frac{m^r/j}{r!} 
		 \leq \frac{e^{2^am}}{j} .
\]

\medskip

(b) By the multinomial theorem and part (a), we have that
\als{
\E \left[ |\sL_m(\X_k)| S(\X_k)^r \right] 
	&=\sum_{a_1+\cdots+a_k=r}
		 \binom{r}{a_1,\dots,a_k}\E \left[|\sL_m(\X_k)|\prod_{j=1}^k (jX_j)^{a_j} \right]  \\
	&\ll_{r,m} \E\left[ |\sL_m(\X_k)| \right]
			\sum_{a_1+\cdots+ a_k=r} \prod_{j=1}^k k^{\max(0,a_j-1)}.
}
Let $J$ be the set of indices $i$ such that $a_i\neq 0$. For each $J$, the product on the right side above is $k^{r-|J|}$, and there are $O_r(1)$ choices for the numbers $a_i$, $i\in J$, with sum $r$. For each $j\in \{1,2,\ldots,r\}$, there are $\le k^j$ subsets $J\subset \{1,\ldots,k\}$ of cardinality $j$. Thus the sum above over $a_1,\ldots,a_r$ is $O(k^r)$, as claimed.
\medskip

(c) We have that
\[
\E \left[ \frac{|\sL_m(\X_k)|}{\max\{C^+(\X_k),k-S(\X_k)\}^r} \right] 
	\ll_r \frac{\E|\sL_m(\X_k)|}{k^r}  + \E \left[ \frac{|\sL_m(\X_k)|}{(C^+(\X_k))^r}  1_{S(\X_k)>k/2} \right] ,
\]
For the second summand, we have that
\als{
\E \left[ \frac{|\sL_m(\X_k)|}{(C^+(\X_k))^r}  1_{S(\X_k)>k/2} \right] 
	&\le \frac{2^{r+1}}{k^{r+1}} \E \left[ \frac{|\sL_m(\X_k)|S(\X_k)^{r+1}}{(C^+(\X_k))^r} 1_{C^+(\X_k)>0} \right] \\
	&= \frac{2^{r+1}}{k^{r+1}} 
		\sum_{\ell=1}^k \frac{1}{\ell^r} \E \left[ |\sL_m(\X_\ell)| S(\X_\ell)^{r+1}  1_{X_\ell\ge 1}\right]  \\
	&\le\frac{2^{r+1}}{k^{r+1}} 
		\sum_{\ell=1}^k \frac{1}{\ell^r} \E \left[ |\sL_m(\X_\ell)| S(\X_\ell)^{r+1} X_\ell \right] ,
}
by Lemma \ref{Lclem-higherdim}(a,b). Now by straightforward modification of the proof in part (b) we have
\[
  \E \left[ |\sL_m(\X_\ell)| S(\X_\ell)^{r+1} X_\ell\right] \ll_{r,m} \ell^r \E |\sL_m(\X_\ell)|,
\]
so
\[
  \E \left[ \frac{|\sL_m(\X_k)|}{(C^+(\X_k))^r}  1_{S(\X_k)>k/2} \right]
  \ll_{r,m} \frac1{k^{r+1}} \sum_{\ell=1}^k \E |\sL_m(\X_\ell)| \leq \frac{\E|\sL_m(\X_k)|}{k^r}.\qedhere
\]
\end{proof}


We also need to recall~\cite[Proposition~2.1]{EFG1}.

\begin{proposition}\label{sieve}
Let $c_1,\dots,c_k$ be nonnegative integers such that $n-S(\c)$ is at least $k+1$. Then the number of $\pi\in\cS_n$ with exactly $c_i$ $i$-cycles for each $i\leq k$ is
\[
	\order \frac{n!}{k\prod_{i=1}^k c_i! i^{c_i}}.
\]
\end{proposition}


We are now ready to prove Proposition~\ref{local-global}. In keeping with the analogy with analytic number theory, in the proof we will speak about ``factorizations'' $\pi = \pi_1 \cdots \pi_m$. By this we mean simply that $\pi$ has fixed sets $A_1,\dots,A_m$ such that $\pi_i = \pi|_{A_i}$ for each $i$. We may think of $\pi_1,\dots,\pi_m$ as partially defined permutations, and we define their cycle types accordingly. Note in this connection that if $\c_i$ is the cycle type of $\pi_i$ then $S(\c_i) = |A_i|$.


\subsection{The lower bound in Proposition~\ref{local-global}}

Recall that $k=k_1\le k_2\le\cdots \le k_m$ and that $k_{m-1}\le ck$. Assume that $n$ is sufficiently large depending on $m$ and $c$. Let $M=\lceil 2e^{2m}\rceil$ and $h=\fl{k/(4M)}$. We also fix integers $L_i=O_{m,c}(1)$ for $i\le m-1$. We focus our attention on permutations $\pi$ factorizing as
\[ 
\pi = \alpha\(\prod_{i=1}^{m-1}\prod_{j=1}^{L_i}\sigma_{ij}\)\beta,
\]
where every cycle of $\alpha$ has length $\leq h$, the total length of $\alpha$ is $|\alpha|<Mh$, each $\sigma_{ij}$ is a cycle of length in the range $Mh<|\sigma_{ij}|<3Mh$, and all cycles of $\beta$ have length $\ge3Mh$. If $\alpha$ is of type $\c=(c_1,\dots,c_h)$ and $|\sigma_{ij}|=\ell_{ij}$ for each $i,j$, then we further assume that
\eq{ellijsL}{
	\(k_i - \sum_{j=1}^{L_i} \ell_{ij}\)_{i=1}^{m-1}\in P_{m-1}\sL_m(\c).
}
This implies that $\pi$ is counted by $i(n,\k)$. Indeed, \eqref{ellijsL} is equivalent to the existence of non-negative integers $(x_{ij})_{i\le m-1,j\le h}$ such that 
\[
k_i = \sum_{j=1}^{L_i} \ell_{ij}+ \sum_{j=1}^h j x_{ij} \quad(1\le i\le m-1)
\]
and $\sum_{i=1}^{m-1}x_{ij}\le c_j$. This means that there are sets $A_1,\dots,A_{m-1}$ of sizes $k_1,\dots,k_{m-1}$, respectively, left invariant by $\pi$. We then define $A_m=\{1,\dots,n\}\setminus\bigcup_{j=1}^{m-1}A_j$, which is also kept invariant by $\pi$ and has size $k_m$. Thus $\pi$ as above is counted by $i(n,\k)$, as claimed.

Now, observe that \eqref{ellijsL} implies that $\sum_{j=1}^{L_i} \ell_{ij}\leq k_i$ for each $i\leq m-1$, so
\[
	n - |\alpha| - \sum_{i=1}^{m-1} \sum_{j=1}^{L_i} \ell_{ij} \geq n - |\alpha| - \sum_{i=1}^{m-1} k_i 
	= k_m - |\alpha| > k_{m-1} - Mh \geq k-Mh \ge 3Mh .
\]
Thus Proposition~\ref{sieve} applies and asserts that the number of such $\pi$ is at least
\eq{numberofpilij}{
 \gg \frac{n!}{L! k \prod_{i,j} \ell_{ij} \prod_{i=1}^h c_i! i^{c_i}} 
 	\gg_{L,m} \frac{n!}{h^{L + 1} \prod_{i=1}^h c_i! i^{c_i}},
}
where $L = \sum_{i=1}^{m-1} L_i$ is the total number of $\sigma_{ij}$.

Fix $\c$ such that $S(\c)\leq Mh$, and suppose $L_i$ and $(\ell_{ij})_{1\leq i\leq m-1, 1\leq j\leq L_i-1}$ have been chosen so that each $\ell_{ij}$ is in the range $Mh<\ell_{ij}<3Mh$ and
\eq{Li}{
	2Mh < k_i - \sum_{j=1}^{L_i-1} \ell_{ij} < 3Mh\qquad(1\leq i\leq m-1).
}
Then, since $S(\c)\leq Mh$, the number of $(\ell_{i,L_i})_{1\leq i\leq m-1}$ satisfying $Mh < \ell_{i,L_i} < 3Mh$ and \eqref{ellijsL} is precisely $|P_{m-1}\sL_m(\c)|=|\sL_m(\c)|$. Since $ck\ge k_i\geq k$ and $h\le k/(4M)$, we can choose $L_i \ll_{m,c} 1$ so that the number of $(\ell_{ij})_{1\leq j\leq L_i-1}$ satisfying \eqref{Li} is $\gg_{m,c} (Mh)^{L_i-1}$. Thus from \eqref{numberofpilij},
\als{
	i(n,\k)
	&\gg_{m,c} 
		\ssum{c_1,\dots,c_h\geq 0\\ S(\c)\leq Mh} 
			\frac{\(\prod_{i=1}^{m-1} h^{L_i-1}\) |\sL_m(\c)|}{h^{L+1} \prod_{i=1}^h c_i! i^{c_i}}\\
	&= \frac{1}{h^m} \ssum{c_1,\dots,c_h\geq 0\\ S(\c)\leq Mh}
		 \frac{|\sL_m(\c)|}{\prod_{i=1}^h c_i! i^{c_i}}\\
	&\order \frac{\E\left[ |\sL_m(\X_h)| 1_{S(\X_h)\leq Mh}\right]}  {h^{m-1}}.
}
To bound this from below, we use the inequality
\[ 
1_{S(\X_h) \leq Mh} \geq 1 - \frac{S(\X_h)}{Mh}.
\] 
By Lemma \ref{lem-sum-gen-higherdim}(a), we have
\[ 
\E \left[ |\sL_m(\X_h)| S(\X_h)  \right]     = 
	\sum_{j = 1}^{h} j \E \left[ |\sL_m(\X_h)| \cdot X_j \right] \leq h e^{2m} \E |\sL_m(\X_h)|,
\]
so
\[
	\E\left[|\sL_m(\X_h)| 1_{S(\X_h)\leq Mh}\right] 
		\geq \(1 - \frac{he^{2m}}{M}\) 
			\E|\sL_m(\X_h)| 
		\geq \frac12 \E|\sL_m(\X_h)|
\]
by our choice of $M$. Thus
\[
  i(n,\k) \gg_{m,c} \frac{\E |\sL_m(\X_h)|}{h^{m-1}}.
\]
The lower bound in Proposition~\ref{local-global} follows from the above inequality and Lemma~\ref{lem555-higherdim}.


\subsection{The upper bound in Proposition~\ref{local-global}} 
Put $k=k_1$ and $K=k_{m-1}$. Suppose that $\pi\in\cS_n$ has invariant sets of sizes $k_1,\dots,k_m$. Then
\[
  \pi=\pi_1 \pi_2 \cdots \pi_m,
\]
where $\pi_i$ is a product of disjoint cycles of total length $k_i$. Fix a permutation $\tau\in\CS_m$ such that $C^+(\pi_{\tau(1)})\le\cdots\le C^+(\pi_{\tau(m)})$ and, for each $i$, choose a cycle $\sigma_i$ of $\pi_{\tau(i)}$ of length $\ell_i = C^+(\pi_{\tau(i)})$. Note then that $\ell_1\leq k$ and $\ell_{m-1}\leq K$. We can then write $\pi$ as a product of disjoint permutations
\[
  \pi = \alpha\alpha'\sigma_1\cdots\sigma_{m-1}\beta,
\]
where $C^+(\alpha)\le \ell_1$, the permutations in $\alpha'$ have lengths in the range $(\ell_1,\ell_{m-1})$, and $C^-(\beta)\ge\ell_{m-1}$, with $\sigma_m$ being one of the cycles of $\beta$. 
If $\c=(c_1,\dots,c_K)$ and $\c'=(c_1',\dots,c_K')$ are the cycle types of $\alpha$ and $\alpha'$, respectively, then 
\eq{zerocj}{
c_{k+1}=\cdots=c_K=0,
}
\eq{zerocjpr}{
c_1'=\cdots=c_{\ell_1}'=0,
}
and also
\eq{elli-sLdbound}{
	(k_{\tau(i)} - \ell_i)_{i=1}^{m-1} \in P_{m-1}\sL_m(\c+\c').
}
Moreover, since all cycles of $\pi_{\tau(1)}$ other than $\sigma_1$ are cycles of $\alpha$, we must have $\ell_1+S(\c) \geq k$. Therefore
\eq{elli-sL1bound}{
	\ell_{m-1}\ge\cdots \ge \ell_1 \geq Q(\c) : = \max\{C^+(\c), k - S(\c) \}.
}
In particular, by~\eqref{zerocjpr} we have
\eq{zerocjpr2}{
c_1'=\cdots=c_{Q(\c)}'=0,
}

We can now show our hand. We will bound the number of choices for $\pi$ by choosing first $\tau\in\CS_m$, then $\c$ such that~\eqref{zerocj} holds, then $\c'$ such that \eqref{zerocjpr2} holds, $(\ell_i)$ such that~\eqref{elli-sL1bound} and \eqref{elli-sLdbound}  hold, and finally disjoint $\alpha,\alpha',\sigma_1,\dots,\sigma_{m-1},\beta$ of total length $n$ such that $\alpha$ has type $\c$, $\alpha'$ has type $\c'$, $\sigma_i$ is a cycle of length $\ell_i$ for each $i$, and every cycle of $\beta$ has length at least $\ell_{m-1}$ and at least one cycle of length $\ell_m$.

Given $\c,\c',\ell_1,\dots,\ell_{m-1}$, by Proposition~\ref{sieve} the number of choices for $\pi = \alpha\alpha'\sigma_1\cdots\sigma_{m-1}\beta$ is
\als{
&\ll	\frac{n!}{\ell_{m-1}} \prod_{j=1}^{K} \frac{1}{(c_j+c_j'+|\{i<m:\ell_i=j\}|)!j^{c_j+c_j'+|\{i<m:\ell_i=j\}|}}\\
&\leq \frac{n!}{\ell_1\cdots \ell_{m-2}\ell_{m-1}^2} \prod_{i=1}^k \frac{1}{c_i!i^{c_i}}
		\prod_{j=\ell_1+1}^K \frac{1}{c_j'!j^{c_j'}} .
}
Thus
\als{
 i(n,\k)
 &\ll \sum_{\tau\in\cS_m} 
 	\ssum{\c,\c' \\ \ell_1,\dots,\ell_{m-1} \\ \eqref{zerocj}, \eqref{zerocjpr2}, \eqref{elli-sL1bound}, \eqref{elli-sLdbound} } 
	\frac1{\ell_1\cdots\ell_{m-2}\ell_{m-1}^2} 
		\prod_{i=1}^k \frac{1}{c_i!i^{c_i}}
		\prod_{j=\ell_1+1}^K \frac{1}{c_j'!j^{c_j'}}    \nn
	&\le m! \ssum{\c,\c' \\ C^+(\c)\le k \\ c'_i=0, i\le Q(\c)}
		\frac{|\sL_m(\c+\c')|}{Q(\c)^m}
		\prod_{i=1}^k \frac{1}{c_i!i^{c_i}}
		\prod_{Q(\c)<j\le K} \frac{1}{c_j'!j^{c_j'}}    \nn
	&\le m! \ssum{\c\\ C^+(\c)\le k} 
		\frac{|\sL_m(\c)|}{Q(\c)^m}
		\prod_{i=1}^k \frac{1}{c_i!i^{c_i}} \ssum{\c'\\c'_i=0,i\le Q(\c)}
		\prod_{Q(\c)<j\le K} \frac{m^{c_j'}}{c_j'!j^{c_j'}}    ,
}
by Lemma \ref{Lclem-higherdim}(a,b). Calculating the sum over $\c'$, we find that
\als{
i(n,\k)
	\ll \ssum{c_1,\dots,c_k}
		\frac{|\sL_m(\c)|}{Q(\c)^{m}}
		\prod_{i=1}^k \frac{1}{c_i!i^{c_i}} \prod_{Q(\c)<j\le K} e^{m/j} 
	&\ll_m K^{m} \Big( \prod_{i=1}^k  e^{1/i} \Big) \E \left[ \frac{|\sL_m(\X_k)|}{Q(\X_k)^{2m}} \right] \\
	&\ll_m \frac{K^{m}}{k^{2m-1}} \E|\sL_m(\X_k)|,
}
by Lemma \ref{lem-sum-gen-higherdim}(c), which proves the upper bound in Proposition~\ref{local-global}.


\section{The upper bound in Theorem~\ref{main}}\label{sec:upper}

We now turn to the upper bound in Theorem~\ref{main}. Having proved our local-global principle Proposition~\ref{local-global}, our aim is now to prove that
\eq{globalupper}{
  \E |\sL_m(\X_k)| \ll_m k^{m-1-\delta_m} (\log k)^{-3/2}.
}
We begin with
\eq{Ldc}{
 \E |\sL_m(\X_k)| \order \frac{1}{k} \sum_{c_1,\ldots,c_k\geq 0} \frac{|\sL_m(\c)|}{\prod_{j=1}^k c_j! j^{c_j}}.
}
If we fix $r=c_1+\cdots+c_k$, then\footnote{To see the equality~\eqref{ca}, associate to each vector $\a$ the vector $\c$ with $c_i$ the number of indices $j$ such that $a_j=i$. Then $\sL_m(\c) = \sL_m^*(\a)$, $\prod_{j = 1}^k j^{c_j} = a_1 \cdots a_r$, and each $\c$ comes from $r!/(c_1!\cdots c_k!)$ different choices of $\a$. If one thinks of $c_1,\dots,c_k$ as representing the number of $j$-cycles for $j\leq k$ in a random permutation $\pi\in\cS_n$ (which is only really valid in the limit $n\to\infty$, with $k$ fixed), then one can think of $a_1,\dots,a_r$ as the lengths of the cycles of length at most $k$, in no particular order.}
\eq{ca}{
 \sum_{c_1+\cdots+c_k=r} \frac{|\sL_m(\c)|}{\prod_{j=1}^k c_j! j^{c_j}} = \frac{1}{r!} \sum_{a_1,\ldots,a_r=1}^k
 \frac{|\sL_m^*(\a)|}{a_1\cdots a_r},
}
where $\sL_m^*(\a)$ is the set of all $m$-tuples
\[
	\(\tsum_{j\in P_1} a_j, \dots, \tsum_{j\in P_m} a_j\)
\]
as $(P_1,\dots,P_m)$ runs over all ordered partitions of $\{1,\dots,r\}$. From~\eqref{Ldc} and~\eqref{ca} we then have
\eq{upper-start}{
\E |\sL_m(\X_k)| \order \frac{1}{k} \sum_r \frac{1}{r!}  \sum_{a_1,\dots,a_r=1}^k
 \frac{|\sL_m^*(\a)|}{a_1\cdots a_r}.
}

The most common way for $|\sL_m^*(\a)|$ to be small is for many of the $a_i$ to be small. To capture this, let $\tilde{a}_1\leq \tilde{a}_2\leq\cdots$ be the increasing rearrangement of the sequence $\a$ (the \emph{order statistics} of $\a$). Following the proof of Lemma \ref{Lclem-higherdim}(c), we find that
\[
  |\sL_m^*(\a)| =   |\sL_m^*(\tilde{\a})|\leq  |\sL_m^*(\tilde{a}_1,\dots,\tilde{a}_j,0,\dots,0)| \cdot m^{r-j} ,
\]
for any $j\in\{0,1,\dots,r\}$. Since $\sL_m^*(\tilde{a}_1,\dots,\tilde{a}_j,0,\dots,0) 
	\subset[0,\tilde{a}_1+\cdots+\tilde{a}_j]^r$, we find that
\eq{Gbound}{
  |\sL_m^*(\a)| \leq G(\a) 
  	:= \min_{0\leq j\leq r} \(1 + \tilde{a}_1 + \cdots + \tilde{a}_j \)^{m-1} m^{r-j} .
}
It is not unreasonable to expect that
\eq{up-approx}{
  \sum_{a_1,\ldots,a_r=1}^k \frac{G(\a)}{a_1\cdots a_r} \sim  \int_{[1,k]^r} \frac{G(\mathbf{t})}{t_1\cdots t_r} d\mathbf{t}
  = (\log k)^r  \int_{[0,1]^r} G(k^{\xi_1},\dots,k^{\xi_r}) d\boldsymbol{\xi},
}
where here we have enlarged the domain of $G$ to include $r$-tuples of positive real numbers. However, $G$ is not an especially regular function and so \eqref{up-approx} is perhaps too much to hope for. The function $G$ is, however, increasing in every coordinate, and we may exploit this to prove an approximate version of \eqref{up-approx}.

\begin{lemma}\label{lemma4.1}
For any $r\geq 1$, we have
\[
\sum_{a_1,\ldots,a_r=1}^k \frac{|\sL_m^*(\a)|}{a_1\cdots a_r} \ll m^r (1+\log k)^r r! \int_{\Omega_r}
\min_{0\leq j\leq r} m^{-j} (1 + k^{\xi_1} + \cdots + k^{\xi_j})^{m-1} d\bxi,
\]
where $\Omega_r = \{ \bxi : 0 \leq \xi_1 \leq \cdots \leq \xi_r \leq 1\}$.
\end{lemma}

\begin{proof}
Write $h_a$ for the harmonic sum $\sum_{j=1}^a 1/j$. Motivated by the equality
\[
  \frac{1}{a} = \int_{\exp(h_{a-1})}^{\exp(h_a)} \frac{dt}{t} ,
\]
define the product sets
\[
 R(\a) = \prod_{i=1}^r \left[ \exp\( h_{a_i-1} \), \exp\( h_{a_i} \) \right].
\]
Then~\eqref{Gbound} implies that
\[
  \sum_{a_1,\ldots,a_r=1}^k \frac{|\sL_m^*(\a)|}{a_1\cdots a_r} \le \sum_{a_1,\ldots,a_r=1}^k \frac{G(\a)}{a_1\cdots a_r}
  =  \sum_{a_1,\ldots,a_r=1}^k G(\a) \int_{R(\a)} \frac{d \mathbf{t}}{t_1\cdots t_r}.
\]
Consider some $\mathbf{t} \in R(\a)$.  Writing $\tilde t_1 \leq \tilde t_2 \leq \dots \leq \tilde t_r$ for the increasing rearrangement of $\mathbf{t}$, and noting that $a_i<a_j$ implies $t_i\le t_j$, we have
\[
  \exp\( h_{\tilde{a}_i-1} \) \le \tilde{t_i} \le  \exp\( h_{\tilde{a}_i} \) \quad (1\leq i\leq r). 
\]
In particular, from the inequality $h_a \geq \log(a+1)$ we see that $\tilde{t}_i \ge \tilde{a}_i$ for all $i$. Hence
\[
  G(\a) \leq \min_{0\le j\le r} (1 + \tilde{t_1} + \cdots + \tilde{t_j})^{m-1} m^{r-j} = G(\mathbf{t})
\]
for all $\mathbf{t}\in R(\a)$. Thus
\begin{align*}
  \sum_{a_1,\dots,a_r=1}^k G(\a) \int_{R(\a)} \frac{d \mathbf{t}}{t_1\cdots t_r}
  &\leq \sum_{a_1,\dots,a_r=1}^k \int_{R(\a)} \frac{G(\mathbf{t})}{t_1\cdots t_r} d \mathbf{t}\\
  &= \int_{[1,\exp(h_k)]^r}
  \frac{G(\mathbf{t})}{t_1\cdots t_r} d \mathbf{t}\\
  &= h_k^r \int_{[0,1]^r} G(e^{\xi_1 h_k}, \dots, e^{\xi_r h_k}) d\bxi.
\end{align*}
The lemma now follows from the symmetry of the integrand and the bound $h_k \leq 1+\log k$.
\end{proof}

\def\Ur{U_r\(\frac{m-1}{\log m} \log k;\, m-1\)}
Having established Lemma~\ref{lemma4.1}, we can finish the proof of~\eqref{globalupper} by quoting~\cite[Lemma~4.4]{kouk1}. Indeed, in the notation of that paper
\[
  \int_{\Omega_r}
\min_{0\leq j\leq r} m^{-j} (1 + k^{\xi_1} + \cdots + k^{\xi_j})^{m-1} d\bxi = \Ur,
\]
and thus by~\eqref{upper-start} and Lemma~\ref{lemma4.1} we have
\[
  \E|\sL_m(\X_k)| \ll_m \frac{1}{k} \sum_r m^r (1+\log k)^r \Ur.
\]
Now, by~\cite[Lemma~4.4]{kouk1} we have
\[
  \Ur \ll \frac{1 + |r-r_*|^2}{(r+1)! (m^{r-r_*}+1)} 
\]
uniformly for $0 \leq r \leq 10(m-1)r_*$, where
\[
  r_* = \left\lfloor \frac{m-1}{\log m} \log k \right\rfloor.
\]
Otherwise, we use the trivial bound (from the $j=0$ term in the minimum)
\[
\Ur\le \frac{1}{r!} .
\]
Therefore
\als{
  k\cdot \E|\sL_m(\X_k)| 
  	&\ll_m 
  	\sum_{0\le r\le r*} \frac{m^r (1+\log k)^r(1+|r-r^*|^2)}{(r+1)!} \\
  	&\quad+ \sum_{r*<r\le 10(m-1)r_*} \frac{m^{r_*} (1+\log k)^r(1+|r-r^*|^2)}{(r+1)!} \\
  	&\quad + \sum_{r> 10(m-1)r_*} \frac{m^r (1+\log k)^r}{r!} \\
	&\ll \frac{m^{r_*} (1+\log k)^{r_*}}{(r_*+1)!} ,
}
since $10(m-1)r_* \ge 5 m(1+\log k)$ for large enough $k$ in terms of $m$. Stirling's formula then completes the proof of \eqref{globalupper} and thus that of the upper bound in Theorem~\ref{main}.


\section{The lower bound in Theorem~\ref{main}}\label{sec:lower}

We now turn to the lower bound in Theorem~\ref{main}. Having proved our local-global principle Proposition~\ref{local-global}, our aim is now to prove that
\eq{globallower}{
  \E |\sL_m(\X_k)| \gg_m k^{m-1-\delta_m} (\log k)^{-3/2}.
}

\subsection{A double application of H\"older's inequality}

We begin as in Section~\ref{sec:upper} with~\eqref{upper-start}, or rather with a slight variant. Let
\[
J = \fl{\log k},
\]
suppose that $\b = (b_j)_{1\leq j\le J}$ is a vector of arbitrary nonnegative integers, set
\[
r= b_1+\cdots+b_J,
\]
and consider that part of the sum in~\eqref{Ldc} in which
\eq{cj-condition}{
 \sum_{i\in[e^{j-1},e^j)} c_i = b_j \quad (1\leq j \leq J), \qquad c_i=0 \ (i \geq e^J).
}
For each $j\ge 1$, $b_j$ represents the number of cycles in the interval $[e^{j-1},e^j)$.
By arguing just as in the derivation of~\eqref{upper-start}, we have
\eq{lower-start}{
  \sum_{\substack{c_1,\dots,c_k\geq0\\\eqref{cj-condition}}} \frac{|\sL_m(\c)|}{\prod_{i=1}^k c_i! i^{c_i}} 
  	= \frac{1}{\prod_j b_j!} \sum_{\a \in \DD(\b)} \frac{|\sL_m^*(\a)|}{a_1\cdots a_r},
}
where
\[
  \DD(\b) = \prod_{j=1}^J [e^{j-1}, e^j)^{b_j} .
\]
that is, the first $b_1$ conponents of $\a \in  \DD(\b)$ are in $[1,e)$ and are otherwise unordered, the next
$b_2$ components of $\a \in  \DD(\b)$ are in $[e,e^2)$, etc.
For fixed $\b\in\Z_{\ge0}^J$ and $s\in\{1,\dots,r\}$,  define $j_s   \in\{1,\dots,J\}$ by 
\[
b_1+\cdots+b_{j_s-1}<s\le b_1+\cdots + b_{j_s} ,
\]
so that if $\a\in\DD(\b)$, then $a_s\in[e^{j_s-1},e^{j_s})$. Finally, let 
\[
\lambda_j = \sum_{e^{j-1}\le a<e^j} \frac{1}{a} = 1 + O(e^{-j}) \qquad (j\ge 1).
\]


\begin{lemma}\label{doubleholder}
For any $\b = (b_1,\dots,b_J)$ and $p\in(1,2]$ we have
\[
  \sum_{\a\in\DD(\b)} \frac{|\sL_m^*(\a)|}{a_1\cdots a_r} 
  \geq \frac{ m^{pr/(p-1)} \prod_{j=1}^J \lambda_j^{2b_j}}{\(\sum_\CP \(\sum_\CQ S(\CP,\CQ)\)^{p-1}\)^\frac1{p-1}},
\]
where the sums run over all ordered partitions $\CP=(P_1,\ldots,P_m)$ and $\CQ=(Q_1,\ldots,Q_m)$ of $\{1,\ldots,r\}$, and $S(\CP,\CQ)$ is the sum of $1/(a_1\cdots a_r)$ over all $\a\in\DD(\b)$ such that $\sum_{s\in P_i} a_s = \sum_{s\in Q_i} a_s$ for each $i=1,\dots,m$.
\end{lemma}


\begin{proof}
Given $\a\in\N^r$ and $\x \in \Z_{\ge0}^m$, let $R(\a, \x)$ be the number of partitions $\CP$ such that $x_i = \sum_{s\in P_i} a_s$ for each $i=1,\dots,m$. Then the support of $R(\a,\x)$ is $\sL_m^*(\a)$, and $\sum_\x R(\a,\x) = m^r$, the total number of partitions $\CP$. Thus, H\"older's inequality yields that
\al{
m^r \prod_{j=1}^J \lambda_j^{b_j} 
    &= \sum_{\a\in\DD(\b)} \sum_{\x\in\sL_m^*(\a)} \frac{R(\a,\x)}{a_1\cdots a_r}   \nn
    &\leq \(\sum_{\a\in\DD(\b)} \frac{|\sL_m^*(\a)|}{a_1\cdots a_r}\)^{1-1/p} 
    	\(\sum_{\a\in\DD(\b)} \sum_\x \frac{R(\a,\x)^p}{a_1\cdots a_r} \)^{1/p}.\label{firstholder}
}
Meanwhile,
\als{
  \sum_\x R(\a,\x)^p
    &= \sum_\x R(\a,\x)^{p-1} \sum_\CP 1_{x_i = \sum_{s\in P_i} a_s~\text{for}~i=1,\dots,m}\\
    &= \sum_\CP R\(\a,\(\tsum_{s\in P_i} a_s\)_i\)^{p-1},
}
so by another application of H\"older's inequality we have
\als{
  \sum_{\a\in\DD(\b)} \sum_\x \frac{R(\a,\x)^p}{a_1\cdots a_r}
    &= \sum_\CP \sum_{\a\in\DD(\b)} \frac{R\(\a,\(\tsum_{s\in P_i} a_s\)_i\)^{p-1}}{a_1\cdots a_r}\\
    &\leq \sum_\CP \(\sum_{\a\in\DD(\b)} \frac{R\(\a,\(\tsum_{s\in P_i} a_s\)_i\)}{a_1\cdots a_r}\)^{p-1} \(\sum_{\a\in\DD(\b)} \frac1{a_1\cdots a_r}\)^{2-p}\\
    &= \sum_\CP \(\sum_\CQ S(\CP,\CQ)\)^{p-1} \prod_{j=1}^J \lambda_j^{b_j(2-p)}.
}
The lemma follows from this and~\eqref{firstholder}.
\end{proof}


\subsection{Bounding the low moment}

Next, fix $\CP$ and $\CQ$ and consider $S(\CP,\CQ)$, the sum of $1/(a_1\cdots a_r)$ over all solutions $\a$ to the linear system
\[
  \sum_{s\in P_i} a_s = \sum_{s\in Q_i} a_s
  	,\qquad(i=1,\dots,m),
\]
or, equivalently,
\eq{ajsystem}{
  \sum_{s\in P_i\setminus Q_i} a_s - \sum_{s\in Q_i\setminus P_i} a_s = 0,
  	\qquad(i=1,\dots,m).
}
In order to bound $S(\CP,\CQ)$ we will in effect upper-triangularize this system. This process admits a convenient combinatorial description. Form a weighted graph $\G$ with vertices $\{1,\dots,m\}$ by placing an edge between $i_1$ and $i_2$ whenever the equations in~\eqref{ajsystem} indexed by $i_1$ and $i_2$ have a variable in common, i.e., whenever
\newcommand{\symdiff}{\vartriangle}
\[
  (P_{i_1}\symdiff Q_{i_1}) \cap (P_{i_2} \symdiff Q_{i_2}) \neq \emptyset,
\]
where $A \symdiff B := (A \cup B) \setminus (A \cap B)=(A\setminus B) \cup (B\setminus A)$.
Then we assign to the edge $e=\{i_1,i_2\}$ the \emph{label}
\[
s_e = \max (P_{i_1}\symdiff Q_{i_1}) \cap (P_{i_2} \symdiff Q_{i_2})
\]
and \emph{weight}
\[
w_e = j_{s_e}. 
\]
Note that if $P_i=Q_i$ for some $i$, then the vertex labeled $i$ is isolated in the graph $\G$.
Also, note that the labels must be distinct, while the weights need not be.
If $I_n$, $1\le n\le N$, are the components of $\CG$, we then find that $P_{i_1}\cap Q_{i_2}=\emptyset$ whenever $i_1\in I_{n_1}$ and $i_2\in I_{n_2}$ for $n_1\neq n_2$. Consequently,
\eq{union-cond}{
  \bigcup_{i\in I_n} P_i = \bigcup_{i\in I_n} Q_i\quad(1\le n\le N) ,
}
so that the more components $\CG$ has, the more relations we have between the partitions $\CP$ and $\CQ$.

For a subgraph $\CH\subset \CG$ (a subset of the vertices and edges of $\G$), we denote by $A(\CH)$ the set of labels occurring in $\CH$. We show in the next lemma that, given a subforest $\CF\subset \G$ (that is to say, an acyclic subgraph of $\G$ or, equivalently, a disjoint union of subtrees of $\G$), the variables $(a_s)_{s\in A(\CF)}$ are determined by $(a_s)_{s\notin A(\CF)}$ and~\eqref{ajsystem}. Moreover, the quality of the bound implied for $S(\CP,\CQ)$ is measured by the total weight of $\CF$.


\begin{lemma}\label{forestbound}
If $\CF$ is a subforest of $\G$, then the variables $(a_s)_{s\in A(\CF)}$ are determined by $(a_s)_{s\notin A(\CF)}$ and~\eqref{ajsystem}. Consequently,
\[
  S(\CP,\CQ) \ll_m e^{-\sum_{s \in A(\CF)} j_s} \prod_{j=1}^J  \lambda_j^{b_j} .
\]
\end{lemma}


\begin{proof} Write $A=A(\CF)$ for convenience. For the first part, first note that for any edge $e=\{i_1,i_2\} \in \CF$, the variable $a_{s_e}$ appears in the equations $\sum_{s\in P_i\setminus Q_i} a_s - \sum_{s\in Q_i\setminus P_i} a_s=0$ for $i=i_1$ and $i=i_2$, and no others, since the sets $P_1,\dots,P_m$ are pairwise disjoint, and the same is true for the sets $Q_1,\dots,Q_m$. Thus, if $i$ is a leaf of $\CF$ and $e$ is the edge of $\CF$ incident with $i$, then, out of all the variables $(a_s)_{s \in W}$, the equation
\[
  \sum_{s\in P_i\setminus Q_i} a_s - \sum_{s\in Q_i\setminus P_i} a_s = 0
\]
involves only $a_{s_e}$, so indeed $a_{s_e}$ is determined by $(a_s)_{s\notin W}$ and~\eqref{ajsystem}. Next, remove $e$ from $\CF$ and continue inductively.

Now, since the variables $(a_s)_{s \in A}$ are determined by $(a_s)_{s\notin A}$ and~\eqref{ajsystem}, it follows that
\begin{align*}
  S(\CP,\CQ)
  = \sum_{\substack{\a\in\DD(\b)\\\eqref{ajsystem}}} \frac1{a_1\cdots a_r}
  \leq \sum_{\substack{a_s\in[e^{j_s-1},e^{j_s})\\ (s\notin A)}} 
  	\frac{1}{\prod_{s\in A} e^{j_s-1} \prod_{s\notin A} a_s}
  &= \frac1{\prod_{s\in A} e^{j_s-1} \lambda_{j_s}} \prod_{j=1}^J \lambda_j^{b_j} \\
  &\ll_m e^{-\sum_{s \in A} j_s} \prod_{j=1}^J  \lambda_j^{b_j}.\qedhere
\end{align*}
\end{proof}


\begin{figure}[t]
\tikzstyle{vertex}=[circle,fill=black,minimum size=5pt,inner sep=0pt]
\tikzstyle{selected vertex} = [vertex, fill=red!24]
\tikzstyle{edge} = [draw,thick,-]
\tikzstyle{weight} = [font=\small]
\tikzstyle{selected edge} = [draw,line width=5pt,-,red!50]
\tikzstyle{ignored edge} = [draw,line width=5pt,-,black!20]

\pgfdeclarelayer{background}
\pgfsetlayers{background,main}
\centering
\begin{tikzpicture}[scale=1.3, auto,swap]
    
    \foreach \pos/\name in {{(0,1)/a}, {(3,0)/b}, {(3,2)/c},
                            {(0,5)/d}, {(3,4)/e}, {(3,6)/f},
                            {(6,0)/g}, {(8,0)/h}, {(6,2)/i}, {(8,2)/j},
                            {(6,4)/k}, {(8,4)/l}, {(6,6)/m}, {(8,6)/n}}
        \node[vertex] (\name) at \pos {};
        
    \foreach \source/ \dest /\weight in {b/c/6, e/f/6, g/h/10, g/j/5, i/j/10, 
                                         m/n/10, k/l/10, k/n/5, k/m/5, n/l/3}
        \path[edge] (\source) -- node[weight] {$\weight$} (\dest);
\end{tikzpicture}
\caption{A graph $\mathcal{G}$ and a heaviest subforest. Edge weights are indicated.}
\end{figure}
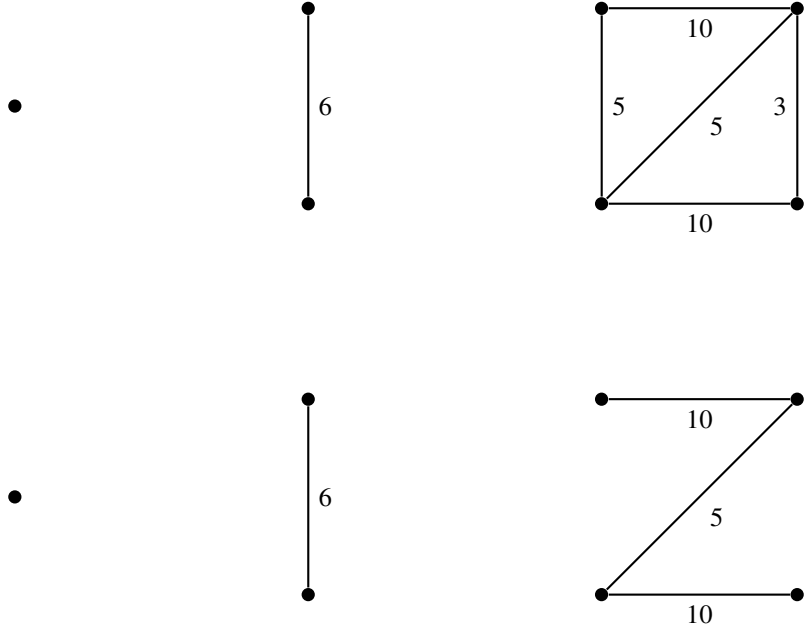


To apply Lemma~\ref{forestbound} most profitably, we should choose a subforest $\CF\subset\G$ which maximizes
the total weight 
\[
W(\CF) := \sum_{s\in A(\CF)} j_s.
\]
Such a $\CF$ will necessarily be a spanning subforest, and thus have the same
number of connected components as $\G$.  See e.g. Figure 1.


\begin{lemma}\label{sumPsumQbound}
\[
\sum_\CP \(\sum_\CQ S(\CP,\CQ)\)^{p-1} \ll_m m^r \(\prod_{j=1}^J \lambda_j^{b_j(p-1)}\)
	 \(1+\sum_{j=1}^J 
	\(m^\frac{p-1}{m-1}\)^{b_{1}+\cdots+b_j} e^{-(p-1)j}\)^{m-1} .
\]
\end{lemma}


\begin{proof} We will consider several graphs throughout the proof, but we fix for all time the vertex set as $\{1,\dots,m\}$. 

Before we begin, we make some observations. Fix, for the moment, two partitions $\CP$ and $\CQ$, and consider the associated weighted graph $\CG$. As noted earlier, a $\CF$ which is
a heaviest subforest is a spanning subforest of $\CG$.  For any $s$, denote by $\CF_s$ the subforest of $\CF$ consisting of all edges $e\in\CF$ with
$s_e \ge s$.  We now show that there is a heaviest subforest $\CF$ with the following property:
whenever $s\in P_i\cap Q_j$, then $i$ and $j$ lie in the same component of $\CF_s$. 
To see this, we separate three cases.
\begin{enumerate}[label={\upshape(\alph*)}]
  \item If $i=j$, then the claim is trivially true. In particular, we note in this case that $i$ is an isolated vertex of $\G$ (and hence of $\CF_s$).
  \item If $i\neq j$ and $s\in A(\CF)$, then in fact $\{i,j\}$ is an edge of $\CF_s$.
  \item Suppose that $i\neq j$ and $s\notin A(\CF)$. As noted before, $i$ and $j$ must lie in the same component of $\G$, hence in the same component of $\CF$.  There is a unique path from $i$ to $j$ 
  within $\CF$.  If $i$ and $j$ do not lie in the same component of $\CF_s$, then this path takes the form
\[
i \rightarrow \cdots \rightarrow i' \overset{s'}\rightarrow j' \rightarrow\cdots \rightarrow j
\]
with at least one label $s'<s$ and weight $j_{s'}$. But then we can create another subforest $\CF'$ by removing the edge $\{i',j'\}$ from $\CF$ (breaking the tree) and adding the edge $\{i,j\}$ (reconnecting the tree), whose label is $\ge \max(P_i\cap Q_j)\ge s$ with substitute weight at least $j_{s} \ge j_{s'}$.
\end{enumerate}

We are now ready to prove the lemma. Given an ordered partition $\CP$, a forest $\CF$ with $N$ components, and a set of labels $A=A(\CF)$ on the edges of $\CF$, write $M(\CP,\CF,A)$ for the number of $\CQ$ for which the associated graph $\G$ has a heaviest subforest $\CF$.    
The above discussion implies that for each $s\in\{1,\ldots,r\}$, 
the number of possibile $j\in \{1,\ldots,m\}$ so that $s\in Q_j$ is at most $|I_{st_s}|$, where 
 $I_{s1},\dots,I_{sN_s}$ denote the components of $\CF_s$ and $t_s$ is defined by $s\in P_i$ and $i\in I_{st_s}$. 
It follows that
\[
M(\CP,\CF,A) 
	\le \prod_{s=1}^r |I_{st_s}|.
\]
Together with Lemma \ref{forestbound} and the inequality $(x+y)^{p-1}\le x^{p-1}+y^{p-1}$, which is true for $p\in[1,2]$ and $x,y\ge0$, we find that
\als{
  \sum_\CP \(\sum_\CQ S(\CP,\CQ)\)^{p-1}
  	&\ll_m \left(\prod_{j=1}^J \lambda_j^{b_j(p-1)}\right) \sum_{\CP} \( \sum_{\CF,A} e^{-W(\CF)} 
             \prod_{s=1}^r  |I_{st_s}| \)^{p-1} \\
&\le \left(\prod_{j=1}^J \lambda_j^{b_j(p-1)}\right) 
  	\sum_{\CF,A} e^{-(p-1)W(\CF)}  
	 \ssum{1\le t_s\le N_s\\ 1\le s\le r}
  	 \ssum{(P_1,\dots,P_m) \\ s\in P_i\, \Rightarrow\, i\in I_{st_s}} 
	 	\prod_{s=1}^r  |I_{st_s}|^{p-1} \\
	 &= \left(\prod_{j=1}^J \lambda_j^{b_j(p-1)}\right) 
  	\sum_{\CF,A} e^{-(p-1)W(\CF)}  \ssum{1\le t_s\le N_s\\ 1\le s\le r}
		|I_{1t_1}|^p\cdots |I_{rt_r}|^p \\
	&= \left(\prod_{j=1}^J \lambda_j^{b_j(p-1)}\right) 
  	\sum_{\CF,A} e^{-(p-1)W(\CF)}  
		\prod_{s=1}^r (|I_{s1}|^p+\cdots+|I_{sN_s}|^p) .
}
Note that
\[
  \max\{x_1^p + \cdots + x_n^p : x_1+\cdots+x_n=m, x_1,\dots,x_n\geq 1\} = (m-n+1)^p + n-1
\]
for $m\geq n$: this follows from convexity of the function $(x_1,\dots,x_n)\mapsto x_1^p + \cdots + x_n^p$, since the maximum of a convex function in a simplex occurs at one of its vertices. Therefore
\als{
 \frac{ \sum_\CP \(\sum_\CQ S(\CP,\CQ)\)^{p-1}}{\prod_{j=1}^J \lambda_j^{b_j(p-1)}}
 	\ll_m \sum_{\CF,A} e^{-(p-1)W(\CF)  }  
		\prod_{s=1}^r ((m-N_s+1)^p+N_s-1)  .
}

Let $f$ denote the number of edges in $\CF$, and write $s_1<\cdots< s_f$ for the edge labels of $F$, which we know are distinct. We also write $s_0=0$ and $s_{f+1}=r$ for convenience. Recall that $N$ is the number of components of $\CF$, so that $N_1=N$. Since a tree of $n$ vertices contains exactly $n-1$ edges, we must have that $f=m-N$.

Note that $N_s=N_{s_i}$ is constant when $s\in(s_{i-1},s_i]$, as well as that $N_{s_i}=\min\{m,N_{s_{i-1}}+1\}$, since the removal of one edge from $F_{s_{i-1}}$ cuts one component into two pieces, creating exactly one additional component in $F_{s_i}$. Consequently, $N_{s_i} = N+i-1$ for $i\le f$ and $N_s=m$ for $s>s_f$, so that
\[
 (m-N_s+1)^p+N_s-1= \begin{cases}
                     (f-i+2)^p+m-f+i-2 & \text{ if } s_{i-1}<s\le s_i; i\le f \\ m & \text{ if } s_f<s\le s_{f+1}=r.
                    \end{cases}
\]
There are $O_m(1)$ forests $\CF$, and $O_m(1)$ orderings of the edges within each forest.
Therefore,
\als{ 
\frac{\sum_\CP \(\sum_\CQ S(\CP,\CQ)\)^{p-1}}{\prod_{j=1}^J \lambda_j^{b_j(p-1)}}
&\ll_m m^r +  \sum_{f=1}^{m-1} \sum_{1\le s_1<\cdots<s_f\le r}
		e^{-(p-1)\sum_{i=1}^f j_{s_i} }   \\
	&\qquad\times \prod_{i=1}^f \( (f-i+2)^p +m-f+i-2 \)^{s_i-s_{i-1}} m^{s_{f+1}-s_f} ,
}
with the summand $m^r$ corresponding to $f=0$, that is to say the forest with no edges. 
Lemma~3.7 in \cite{kouk1} implies that
\[
(\ell+1)^p+m-\ell-1  \le  m \(m^\frac{p-1}{m-1}\)^\ell  \quad(0\le \ell\le m-1) ,
\]
provided that $p$ is sufficiently close to $1$ in terms of $m$, so that
\als{
\frac{\sum_\CP \(\sum_\CQ S(\CP,\CQ)\)^{p-1}}{m^r \prod_{j=1}^J \lambda_j^{b_j(p-1)}}
	&\ll_m 1 + \sum_{f=1}^{m-1} 
		\sum_{1\le s_1<\cdots<s_f\le r} e^{-(p-1)(j_{s_1}+\cdots+j_{s_f})} \(m^\frac{p-1}{m-1}\)^{s_1+\cdots+s_f} \\
	&\le 1+  \sum_{f=1}^{m-1} \left( \sum_{s=1}^r e^{-(p-1)j_s} 
		\(m^\frac{p-1}{m-1}\)^s  \right)^f ,
}
by unordering the summands. Clearly, if the expression which we raise to the $f$-th power is $<1$, then the term 1 dominates; otherwise, the term with $f=m-1$ dominates. In any case,
\[
\frac{\sum_\CP \(\sum_\CQ S(\CP,\CQ)\)^{p-1}}{m^r \prod_{j=1}^J \lambda_j^{b_j(p-1)}}
  	\ll_m 1+ \left( \sum_{s=1}^r e^{-(p-1)j_s} 
		\(m^\frac{p-1}{m-1}\)^s  \right)^{m-1} .
\]
In order to complete the proof, note that
\als{
\sum_{s=1}^r e^{-(p-1)j_s} 
		\(m^\frac{p-1}{m-1}\)^s  
	&= \sum_{j=1}^J e^{-(p-1)j} \sum_{b_1+\cdots+b_{j-1}<s\le b_1+\cdots+b_j}\(m^\frac{p-1}{m-1}\)^s  \\
	&\ll \sum_{j=1}^J e^{-(p-1)j} \(m^\frac{p-1}{m-1}\)^{b_1+\cdots+b_j} .
}
The claimed estimate then follows.
\end{proof}


By combining~\eqref{Ldc} and~\eqref{lower-start} with Lemmas~\ref{doubleholder} and~\eqref{sumPsumQbound}, we have proved the following proposition.


\begin{proposition}\label{generallower}
If $p\in(1,2]$ is sufficiently close to $1$ in terms of $m$, then
\[
  \E|\sL_m(\X_k)| \gg_m \frac1k \sum_{b_1,\dots,b_J\geq 0} 
  	\frac{m^r \prod_{j=1}^J \lambda_j^{b_j}/b_j!}{\(1+\sum_{j=1}^J \(m^\frac{p-1}{m-1}\)^{b_{1}+\cdots+b_j} e^{-(p-1)j}\)^\frac{m-1}{p-1}}.
\]
\end{proposition}


\emph{Remark.} The analysis given in this subsection differs technically from the corresponding analysis in~\cite{kouk1}. First of all, the combinatorial language of trees and forests used to describe the interdependencies in the relevant linear system is new, but even when both arguments are cast in this language, there is a difference, related to how we analyze the partitions giving rise to a particular heaviest subforest. The difference is parallel to that between two of the best known algorithms for finding a minimal spanning tree, namely Prim's algorithm, which builds a tree by repeatedly adding the least expensive edge growing out of the current tree, and Kruskal's algorithm, which builds a forest by repeatedly adding the least expensive edge which does not create a cycle. In \cite{kouk1}, the argument is more closely related to Prim's algorithm, while the argument here is more closely related to Kruskal's algorithm.


\subsection{Input from order statistics}

Now fix $r = \frac{m-1}{\log m} J + O(1)$, and let $\B = \B_{C,C'}$ be the set of all $\b = (b_1,\dots,b_J)$ such that
\begin{enumerate}[label={\upshape(\alph*)}]
  \item $b_1+\cdots+b_J=r$;
  \item $b_j = 0$ for every $j \leq C$;
  \item $b_1+\cdots+b_j  \leq Cj$ for every $j\le J$;
  \item we have the bound
  	\[
	\sum_{j=1}^J \(m^\frac{p-1}{m-1}\)^{b_1+\cdots+b_j} e^{-(p-1)j} \leq C'.
	\]
\end{enumerate}
Here, $C$ and $C'$ are two integers which we will choose to be sufficiently large depending only on $m$. In this case, Proposition~\ref{generallower} implies
\als{
  \E|\sL_m(\X_k)|
  &\gg_m \frac{m^r}{k} \sum_{\b\in\B} \frac{\prod_{j>C} \(1-O(e^{-j})\)^{Cj}}{b_{C+1}!\cdots b_J!}\\
  &\gg_m \frac{m^r}{k} \sum_{\b\in\B} \frac1{b_{C+1}!\cdots b_J!},
}
where the second inequality holds just because the product is convergent and we can choose $C$ sufficiently large.

Let $R(\b)$ be the set of all $\bxi\in[0,1]^r$ such that $0\leq \xi_1\leq \cdots \leq \xi_r< 1$ and such that, for each $j\in\{1,\dots,J-C\}$, exactly $b_{j+C}$ of the variables $\xi_s$ are such that
\[
\frac{j-1}{J-C}\le  \xi_s< \frac{j}{J-C} .
\]
Then
\[
  \sum_{\b\in\B} \frac1{b_{C+1}!\cdots b_J!}  = \sum_{\b\in\B} (J-C)^r \Vol(R(\b)) = (J-C)^r \Vol(\cup_{\b\in\B} R(\b)) \geq (J-C)^r \Vol(Y),
\]
where $Y$ is the set of all $\bxi\in[0,1]^r$ such that $0\leq \xi_1\leq\cdots\leq\xi_r< 1$, $\xi_s\ge(s-C^2)/(CJ-C^2)$ for each $s$, and
\[
 e^{-1}+ \sum_{s=1}^r \(m^\frac{p-1}{m-1}\)^s e^{-(p-1)(J-C)\xi_s} \leq 
  C'(1-e^{1-p}) .
\]
If $C$ is large enough in terms of $m$, and then $C'$ is sufficiently large in terms of $C$, $p$ and $m$, then \cite[Lemma~3.10]{kouk1} implies that
\[
  \Vol(Y) \gg \frac1{r\cdot r!},
\]
It follows from this and a short calculation using Stirling's formula that
\als{
  \E|\sL_m(\X_k)|
	  \gg_m \frac{m^r}{k} \frac{J^r}{r\cdot r!}
  	\order_m k^{m-1-\delta_m} (\log k)^{-3/2}.
}
The lower bound in Theorem~\ref{main} is now a direct corollary of this estimate and of Proposition~\ref{local-global}.


\section{Imprimitive transitive subgroups}\label{sec:imprimitive}

In this section we use Theorem~\ref{main} to prove Theorem~\ref{imprimitive} by fleshing out the argument outlined in Section~\ref{sec:outline}. We will start with bounded $\nu$ and gradually treat larger and larger $\nu$. To bound $I(n)$ we will then use the trivial bound
\begin{equation}\label{trivInInnubound}
  I(n) \leq \ssum{\nu\mid n \\  1<\nu<n} I(n,\nu).
\end{equation}

\subsection{Small $\nu$}

We did most of the work for the case in which $\nu$ is bounded already in Section~\ref{sec:outline}. We state the conclusion here.

\begin{proposition}\label{smallnu}
Let $\nu$ be a bounded divisor of $n$. Then
\[
I(n,\nu) \order_\nu
  \begin{cases}
    n^{-\delta_\nu}(\log n)^{-3/2} & \text{if}~1<\nu\leq 4,\\
    n^{-1+1/(\nu-1)} & \text{if}~\nu\geq 5.
  \end{cases}
\]
\end{proposition}
\begin{proof}
This follows immediately from Lemma~\ref{partition-d-1}, Theorem~\ref{main}, and Proposition~\ref{dominantpart}. Specifically, by Proposition~\ref{dominantpart} we know that
\[
  I(n,\nu) \order_\nu i(n,(d_in/\nu)_i,(d_i)_i),
\]
where
\[
  \d =
  \begin{cases}
    (1,\dots,1) & \text{if}~\nu\leq 4,\\
    (\nu-1,1) & \text{if}~\nu\geq 5.
  \end{cases}
\]
Theorem~\ref{main} provides an estimate for $i(n,(d_in/\nu)_i,(d_i)_i)$ for $\d=(1,\dots,1)$, while Lemma~\ref{partition-d-1} provides an estimate for $i(n,(d_in/\nu)_i,(d_i)_i)$ for $\d = (\nu-1,1)$.
\end{proof}

\subsection{Intermediate $\nu$}

For unbounded but not too large $\nu$ our goal is still to prove $I(n,\nu)\order n^{-1+1/(\nu-1)}$. 
As long as $\nu$ is less than $\log n$, this is not the same as $n^{-1}$, and as long as $\nu$ is less than $(\log n)^{1/2}$, this is not the same as $n^{-1+1/\nu}$, so we must continue to give special status to the partition $(\nu-1,1)$. 

\begin{proposition}\label{intermediatenu}
Let $\nu$ be a divisor of $n$ such that $1000\leq \nu\leq n/\log^2 n$. Then 
\[
  I(n,\nu) \order n^{-1+1/(\nu-1)}.
\]
\end{proposition}

\begin{proof}
The lower bound is immediate from Lemma~\ref{partition-d-1}:
\[
  I(n,\nu) \geq i(n,((\nu-1)n/\nu,n/\nu),(\nu-1,1)) \order n^{-1+1/(\nu-1)}.
\]
Thus it suffices to prove the upper bound.

Consider a partition $(d_i)$ of $\nu$ into $m$ parts, where $d_1\le d_2 \le \cdots \le d_m$. If $m=1$, we have $i(n,(n),(\nu)) \order n^{-1+1/\nu}$ by Lemma \ref{first-lemma}(a). If $m=2$, $d_1=1$ and $d_2=\nu-1$ we get $i(n,((\nu-1)n/\nu,n/\nu),(\nu-1,1)) \order n^{-1+1/(\nu-1)}$ as above. We will show that the sum of all other terms $i(n,(d_i n/\nu)_i,(d_i)_i)$ is $O(n^{-1})$, which will prove the lemma. We will use Lemma~\ref{first-lemma}(b), together with Lemma~\ref{first-lemma}(c) for the parts $d_i\ge 2$, and the main result of \cite{EFG1} for the parts $d_i=1$, namely $i(2n/\nu,(n/\nu,n/\nu)) \le (cn/\nu)^{-\delta_2}$ for some absolute constant $c\in(0,1]$. Writing $\lambda$ for the number of $i<m$ such that $d_i=1$, we find that
\eq{idi}{
 i(n,(d_i n/\nu)_i,(d_i)_i) 
 	&\le \pfrac{d_m n}{\nu}^{-1+1/d_m} 
 	\left(\prod_{i<m,d_i\ge 2} \pfrac{d_i n}{\nu}^{-1+1/d_i} \right) 
 	\pfrac{cn}{\nu}^{-\fl{\lambda/2} \delta_2} \\
	&\le \pfrac{n}{m}^{-1+m/\nu}
 	\left(\prod_{i<m,d_i\ge 2} \pfrac{d_i n}{\nu}^{-1+1/d_i} \right) 
 	\pfrac{cn}{\nu}^{-\fl{\lambda/2} \delta_2} ,
}
where we used the fact that $d_m\ge \nu/m$.

We also make use of the following estimate:
\begin{equation}\label{sumd}
  \sum_{2\le d\le D} \frac{1}{(xd)^{1-1/d}} \ll \frac{1}{x^{1/2}} \qquad (x \ge \log^2 D).
\end{equation}
This is easily proved by observing that the term $d=2$ is $1/(2x)^{1/2}$, while the terms with $d>2$ contribute at most
\[
 \frac{\max_d d^{1/d}}{x} \( x^{1/3} \sum_{d\le \log x}\frac{1}{d}  + e \sum_{\log x < d\le D} \frac{1}{d} \)
 			 \ll \frac{x^{1/3}\log\log x+\log D}{x}\ll \frac{1}{x^{1/2}}.
\]

Now, we use the above discussion to bound the contribution of the remaining terms $i(n,(d_i n/\nu)_i,(d_i)_i)$. First, we deal with those terms that have $m=2$. Relations \eqref{idi} and \eqref{sumd} imply that
\begin{align*}
\sum_{2\le d_1\le \nu/2} i(n,(n-d_1n/\nu,d_1n/\nu),(d_1,\nu-d_1)) 
	&\le \sum_{2\le d_1\le \nu/2} (n/2)^{-1+2/\nu}
		 \pfrac{d_1 n}{\nu}^{-1+1/d_1} \\
	&\ll n^{-1+2/\nu} \pfrac{n}{\nu}^{-1/2} \ll \frac{1}{n\log n}
\end{align*}
uniformly in  $\nu$.

Now consider the partitions with a fixed number $m\in [3,\nu]$ of parts. Applying again \eqref{idi} and \eqref{sumd}, we get that
\begin{align*}
 \ssum{d_1\le \cdots \le d_m \\ d_1+\cdots+d_m=\nu} i(n,(d_i n/\nu),(d_i)) 
 	&\le \pfrac{n}{m}^{-1+m/\nu} \sum_{\lam=0}^{m-1}
		 \pfrac{cn}{\nu}^{-\fl{\lam/2}\delta_2} 
		 	\Bigg( \sum_{2\le d\le \nu} \pfrac{dn}{\nu}^{-1+1/d} \Bigg)^{m-1-\lam} \\
	 &\le \pfrac{n}{m}^{-1+m/\nu} \sum_{\lam=0}^{m-1}  
	 	\pfrac{cn}{\nu}^{-\fl{\lam/2}\delta_2} 
	 	\(O\(\frac{\nu}{n}\)\)^{\frac{m-1-\lam}{2}}\\
	 &\ll  \pfrac{n}{m}^{-1+m/\nu} e^{O(m)} \pfrac{n}{\nu}^{-\frac{m-2}{2}\delta_2}.
\end{align*}
This estimate suffices, and we simply need to analyze the right hand side, denoted $R_m$, in different ranges of $m$ and $\nu$. When $3\le m \le 40$, $R_m\ll n^{-1}$ uniformly in $\nu \geq 1000$. When $40 \le m\le \nu/\log n$ (in particular, $\nu\ge 40\log n$), $n^{m/\nu}\ll 1$ and
$R_m\ll n^{-1} [ O( (n/\nu)^{-\delta_2/2})]^{m-2}$.  Thus 
$$\sum_{40\le m\le \nu/\log n} R_m\ll n^{-1} \pfrac{n}{\nu}^{-19\delta_2}\ll \frac{1}{n}.$$
Finally, suppose $m> \max(40,\nu/\log n)$.  Then $R_m \ll  [ O( (n/\nu)^{-\delta_2/2})]^{m-2}$.
If $\nu< n^{1/4}$, summing on $m$ gives a total of $\sum_{m} R_m \ll (n/\nu)^{-19\delta_2} \ll n^{-1.2}$. On the other hand, if $\nu>n^{1/4}$, then $m>n^{1/4}/\log n$, and we get $\sum_m R_m \ll n^{-100}$.
\end{proof}

\subsection{Large $\nu$}

Our tools are not well adapted to the range $n/\log^2 n\leq \nu<n$, but fortunately those of Diaconis, Fulman, and Guralnick~\cite{dfg08} are. The argument in this subsection is related to~\cite[Theorems~6.3 and~7.4]{dfg08}, but involves a slightly more careful analysis.

We need a small lemma before continuing.
\begin{lemma}\label{L2lem}
The coefficient of $z^m$ in $\exp\(\sum_{k=1}^\infty \frac{z^k}{k^2}\)$ is bounded by $O(1/m^2)$.
\end{lemma}
\begin{proof}
Write
\[
  h(z) = \exp\(\sum_{k=1}^\infty \frac{z^k}{k^2}\) = \sum_{m=0}^\infty c_m z^m.
\]
Clearly the coefficients of $h(z)$ are bounded by those of
\[
  \exp\(\sum_{k=1}^\infty \frac{z^k}{k}\) = \frac1{1-z} = \sum_{m=0}^\infty z^m,
\]
so at least we know $c_m\leq 1$. Moreover, the identity 
\[
  h'(z) = h(z) \sum_{k=1}^\infty \frac{z^{k-1}}{k}
\]
implies the recurrence
\begin{equation}\label{cmrecurrence}
  c_m = \frac1m \sum_{k=1}^m \frac1k c_{m-k}.
\end{equation}
Inserting the trivial bound $c_{m-k}\leq 1$, we deduce that
\[
  c_m \leq \frac1m \sum_{k=1}^m \frac1k \ll \frac{\log m}{m}.
\]
Now, we can insert the bound $c_{m-k} \ll \log(m-k)/(m-k)$ into~\eqref{cmrecurrence} to find that
\[
  c_m \ll \frac1m \sum_{k=1}^m \frac{\log(m-k)}{k(m-k)} \ll \frac{(\log m)^2}{m^2}.
\]
Using~\eqref{cmrecurrence} one more time we obtain
\[
  c_m \ll \frac1m \sum_{k=1}^m \frac{\log(m-k)^2}{k(m-k)^2} \ll \frac1{m^2}.\qedhere
\]
\end{proof}

\begin{proposition}\label{largenu}
If $\nu$ is a divisor of $n$ in the range $n^{1/2} \leq \nu < n$, then $I(n,\nu) \order n^{-1+\nu/n}$.
\end{proposition}

\begin{proof} Set $s = n/\nu$. The lower bound $I(n,\nu) \gg n^{-1+1/s}$ follows trivially from the observation that any permutation all of whose cycle lengths are divisible by $s$ preserves a system of $n/s$ blocks of size $s$, so it suffices to prove the upper bound.

By~\cite[Theorem~6.3(1)]{dfg08}, 
$I(n,\nu)$ is bounded by the coefficient of $z^\nu$ in
\[
  f(z) = \exp\left(\sum_{k=1}^\infty \frac{z^k}{s!} \frac1k \left(\frac{1}{k} + 1\right)\left(\frac{1}{k} + 2\right)\cdots\left(\frac{1}{k} + s-1\right)\right).
\]
Consider for a moment the polynomial
\[
  p(x) = \frac1{(s-1)!} (x+1)(x+2)\cdots(x+s-1).
\]
Clearly, $p$ has nonnegative coefficients, $p(0) = 1$, and $p(1) = s$. In particular, $p$ is a convex function in $[0,1]$, and we deduce that
\[
  p(x) \leq 1 + (s-1)x \leq 1 + sx \quad(0\le x\le 1) .
\]
Inserting $x=1/k$, we find that
\[
  p(1/k) = \frac1{(s-1)!} \(\frac1k + 1\) \cdots \(\frac1k + s-1\) \leq 1 + \frac{s}{k}.
\]
Thus the coefficients of $f(z)$ are bounded by those of
\[
  g(z) = \exp\(\sum_{k=1}^\infty \frac{z^k}{sk} \(1+\frac{s}{k}\)\) = (1-z)^{-1/s} \exp\(\sum_{k=1}^\infty \frac{z^k}{k^2}\).
\]

Now, for $m>0$, the coefficient of $z^m$ in $(1-z)^{-1/s}$ is
\begin{align*}
  (-1)^m \binom{-1/s}{m}
  &= \frac1{m!} \frac1s \(\frac1s + 1\) \cdots \(\frac1s + m - 1\)\\
  &= \frac1{ms} \prod_{j=1}^{m-1} \(1 + \frac1{js}\)\\
  &\order \frac{1}{m^{1-1/s} s},
\end{align*}
(cf.~the calculation in the proof of Lemma~\ref{first-lemma}(a)), while the coefficient of $z^0$ is of course $1$. 
On the other hand, Lemma~\ref{L2lem} implies that the coefficient of $z^m$ in $\exp\(\sum_{k=1}^\infty \frac{z^k}{k^2}\)$ is $O(1/m^2)$, so that
\begin{align*}
  I(n,\nu)
  &\ll \sum_{m=1}^{\nu-1} \frac1{m^{1-1/s} s} \frac1{(\nu-m)^2} + \frac1{\nu^2}\\
  & \order \sum_{m>\nu/2} \frac1{\nu^{1-1/s} s(\nu-m)^2} + \sum_{m\leq \nu/2} \frac1{m^{1-1/s} s \nu^2} + \frac1{\nu^2}\\
  & \order \frac1{\nu^{1-1/s} s} + O\(\frac{\log \nu}{\nu^{2-1/s}s}\) + \frac1{\nu^2}\\
  & \order \frac1{n^{1-1/s}} + \frac{s^2}{n^2}.
\end{align*}
Since $s\leq n^{1/2}$, this implies that $I(n,\nu) \ll n^{-1+1/s}$, as claimed.
\end{proof}

Theorem~\ref{imprimitive} follows immediately from Propositions~\ref{smallnu}, \ref{intermediatenu}, and \ref{largenu}, the bound~\eqref{trivInInnubound}, and the divisor bound, which states that the number of divisors $\nu$ of $n$ is bounded by $n^{O(1/\log\log n)}$.

\begin{remark}\label{lb-rem}
In general, some extra term in our estimate for $I(n)$ is necessary; that is, it is not always true that $I(n) \ll I(n,p)$.  Let $p_1,\ldots,p_k$ be the prime factors of $n$, and consider the set of numbers $m=p_i h$ with $p_i\le \sqrt{n}$, $h\le \frac12 \sqrt{n}$ and $(h,n)=1$.  Such numbers are clearly all distinct. Also, a permutation which is the product of an $m$-cyle and an $(n-m)$-cycle partitions $\{1,\ldots,n\}$ into $n/p_i$ blocks of size $p_i$. The number of such permutations is $\frac{n!}{m(n-m)} \ge \frac{(n-1)!}{m}$, and so we get that
\[
  I(n) \ge \Big( \sum_{p_i\le \sqrt{n}} \frac{1}{p_i} \Big) \Bigg( \ssum{h\le \frac12 \sqrt{n} \\ (h,n)=1} \frac{1}{h} \Bigg) \frac1n.
\]
Now take $n=p_1\cdots p_k$ with $\log n < p_1 < \cdots < p_k < 10\log n$ and $k\order \frac{\log n}{\log\log n}$.  The sum on $p_i$ is $\order \frac{1}{\log\log n}$ and the sum on $h$ is at least
\als{
\sum_{h\le \frac12 \sqrt{n}} \frac{1}{h} - \sum_{i=1}^k \ssum{h\le \frac12\sqrt{n}\\p_i|h}\frac{1}{h} &\ge
\Big( \sum_{h\le \frac12 \sqrt{n}} \frac{1}{h} \Big) \Big( 1 - \sum_{i=1}^h \frac{1}{p_i} \Big) \\
&= \Big( \sum_{h\le \frac12 \sqrt{n}} \frac{1}{h} \Big) \Big( 1 - O\Big( \frac{1}{\log\log n} \Big)\Big) \sim \frac{\log n}{2}.
}
Also, by Theorem~\ref{imprimitive}, $I(n,p_1) \order n^{-1}$.  Hence
\[
 I(n) \gg \frac{\log n}{\log\log n} n^{-1} \gg \frac{\log n}{\log\log n} I(n,p).
\]
\end{remark}


\section{Primitive subgroups}\label{sec:primitive}

We start by recalling the definition of wreath product. The reader may refer to \cite[Chapter 7]{rotman} for more details. Let $D$ and $Q$ be groups with $Q$ acting on some set $\Omega$. Then $Q$ acts on the set of functions $D^\Omega$ via the operation\footnote{Note that there is a typo in the definition of this action in \cite{rotman}.}
\[
q\cdot (d_\omega)_{\omega\in \Omega}:= (d_{q^{-1}\omega})_{\omega\in \Omega} .
\] 
Then we define the wreath product of $D$ and $Q$, denoted by $D\wr Q$, as the semidirect product of $D^\Omega$ and $Q$. More precisely, $D\wr Q=D^\Omega\times Q$ equipped with the operation 
\[
((d_\omega)_{\omega\in \Omega},q) \cdot ((e_\omega)_{\omega\in \Omega},r) 
	:=((d_\omega e_{q^{-1}\omega})_{\omega\in \Omega},qr) .
\]
If $D$ also acts on some set, say $\Lambda$, then $D\wr Q$ acts on $\Lambda^\Omega$ via the operation
\[
((d_\omega)_{\omega\in \Omega},q) \cdot (\lambda_\omega)_{\omega\in\Omega} 
	:= (d_{\omega} \lambda_{q^{-1}\omega})_{\omega\in \Omega} .
\]
(There is also a natural action of $D\wr Q$ on $\Lambda\times\Omega$, defined by $((d_\omega)_{\omega\in \Omega},q) \cdot (\lambda,\tilde{\omega}) 	:= (d_{q\tilde{\omega}} \lambda , q\tilde{\omega})$, but this action is generically imprimitive, so it will not concern us here.) Moreover, this action is faithful if the actions of $D$ on $\Lambda$ and $Q$ on $\Omega$ are so (and $|\Lambda|\geq 2$), in which case $D\wr Q$ can be realized as a subgroup of $\CS_{\Lambda^\Omega}$. In the special case when $D=\CS_a$, $Q=\CS_b$, $\Lambda=\{1,\dots,a\}$ and $\Omega=\{1,\dots,b\}$, we find that $\CS_a\,\wr\CS_b$ is a transitive subgroup of $\CS_{a^b}$. 

\medskip

We need one last definition: given a nontrivial subgroup $G$ of $\cS_n$, the \emph{minimal degree} of $G$ is the smallest number of points moved by a nontrivial element of $G$. Obviously if $1\neq H\leq G$ then the minimal degree of $G$ is at most that of $H$. 

\medskip

We will combine the following two results.

\begin{theorem}[Bovey~\cite{bovey}]\label{bovey}
Let $\alpha\in(0,1)$. If we choose $\pi$ from $\CS_n$ uniformly at random, then the probability that $\pi\neq1$ and $\langle\pi\rangle$ has minimal degree at least $n^\alpha$ is $\le n^{-\alpha+o_\alpha(1)}$.
\end{theorem}

\begin{theorem}[Liebeck--Saxl~\cite{liebecksaxl}]\label{liebecksaxl}
Let $G$ be a primitive subgroup of $\cS_n$ of minimal degree less than $n/3$. Then there are positive integers $m,k,r$ with $m\geq 5$ for which $n=\binom{m}{k}^r$ and $\cA_m^{\times r}\leq G\leq \cS_m\wr \cS_r,$ where $\cS_m$ acts on the $k$-sets of $\{1,\dots,m\}$ and $\cS_m\wr\cS_r$ acts on $r$-tuples of $k$-sets of $\{1,\dots,m\}$.
\end{theorem}

In fact the constant $1/3$ in this theorem can be improved to $3/7$, and even to $1/2$ with explicit exceptions: see Guralnick and Magaard~\cite{guralnickmagaard}. However we only need the following corollary.

\begin{corollary}\label{liebecksaxlcor}
Let $G$ be a primitive subgroup of $\cS_n$ of minimal degree at most $n^{1-\eps}$, and assume that $n$ is sufficiently large depending on $\eps$. Then there are positive integers $m,k,r$ with $k,r\ll_\eps 1$ such that
$\cA_m^{\times r}\leq G\leq \cS_m\wr \cS_r,$ with the action described in Theorem~\ref{liebecksaxl}. In particular, one of the following alternatives holds:
\begin{romenumerate}
	\item $G = \cS_n$ or $\cA_n$;
	\item $G\leq \cS_m$, where $\cS_m$ acts on $k$-sets of $\{1,\dots,m\}$, $n=\binom{m}{k}$, and $1<k\ll_\eps 1$; or
	\item $G\leq \cS_m \wr \cS_r$, where $\cS_m \wr \cS_r$ acts on $\{1,\dots,m\}^r$, $n=m^r$, and $1<r\ll_\eps 1$.
\end{romenumerate}
\end{corollary}
\begin{proof}
Let $\Delta$ be the set of $k$-sets in $\{1,\dots,m\}$. We must show that the minimal degree of $\cS_m\wr\cS_r$ acting on $\Delta^r$ is at least $n^{1-\eps}$ unless $k,r\ll_\eps 1$. Let $g=(\pi_1,\dots,\pi_r;\sigma)\in\cS_m \wr \cS_r$. We note that an $r$-tuple $(A_1,\dots,A_r)\in\Delta^r$ is a fixed point of $g$ if, and only if, 
\eq{fixed-point}{
\pi_j(A_{\sigma^{-1}(j)}) = A_j \quad(1\le j\le r).
}
We separate two cases.

First, suppose that $\sigma\neq 1$. In particular, $\sigma$ has a cycle of length $s>1$, say $(1\cdots s)$. We then find that $g$ respects the decomposition $\Delta^r = \Delta^s \times \Delta^{r-s}$, and $g$ has at most $\binom{m}{k}$ fixed points in its action on $\Delta^s$: if we know $A_1$ and $\pi_1,\dots,\pi_s$, then $A_2,\dots,A_s$ are determined by the relations \eqref{fixed-point}. Thus $g$ has at most
\begin{equation}\label{gnontrivSr}
	\binom{m}{k}\binom{m}{k}^{r-s} \leq \binom{m}{k}^{r-1}
\end{equation}
fixed points in its action on $\Delta^r$.

On the other hand if $\sigma=1$, then $g$ fixes the point $(A_1,\dots,A_r)\in\Delta^r$ if and only if $\pi_i$ fixes $A_i$ for each $i$. Clearly then the greatest number of points are fixed by an element of the form $(\pi_1,1,\dots,1)$ with $\pi_1\neq 1$. Find $x\in\{1,\dots,m\}$ such that $\pi_1(x)\neq x$. Consequently, if $\pi_1$ fixes $A$, then either $x,\pi_1(x)\in A$ or $x,\pi_1(x)\notin A$. We thus find that the number of fixed points of $g$ acting on $\Delta^r$ is at most
\[
	\(\binom{m-2}{k} + \binom{m-2}{k-2}\)\binom{m}{k}^{r-1},
\]
and, as a matter of fact, exactly that if $\pi_1$ is a transposition. By comparing with~\eqref{gnontrivSr}, we see that the greatest number of points are fixed by a transposition in one coordinate in the base, so the minimal degree of $\cS_m\wr\cS_r$ acting on $\Delta^r$ is
\[
	\binom{m}{k}^r - \(\binom{m-2}{k} + \binom{m-2}{k-2}\)\binom{m}{k}^{r-1} 
		= \frac{2k(m-k)}{m(m-1)} \binom{m}{k}^r \geq \frac2{m-1} \binom{m}{k}^r.
\]
This is at least $\binom{m}{k}^{r(1-\eps)}$ unless $k,r\ll_\eps 1$.

The last part of the corollary follows by assigning the case $k=r=1$ to (i), the case $k>1, r=1$ to (ii), and the case $r>1$ to (iii). In the last case we must replace $m$ by $m'=\binom{m}{k}$.
\end{proof}

We need a couple lemmas to help rule out cases (ii) and (iii) of Corollary~\ref{liebecksaxlcor}.

\begin{lemma}\label{caseii}
If $k\geq 2$ then every $\pi \in \cS_m\setminus\{1\}$ has $\gg_k m^{1/2}$ cycles in its action on the set of $k$-sets of $\{1,\dots,m\}$.
\end{lemma}
\begin{proof}
Write $\Omega$ for $\{1,\dots,m\}$ and $\binom{\Omega}{k}$ for the set of $k$-sets of $\Omega$. Either there are at least $m^{1/2}$ disjoint cycles in $\Omega$, or there is a cycle of length at least $m^{1/2}$. In the former case we get at least one cycle in $\binom{\Omega}{k}$ for each choice of $k$ distinct cycles in $\Omega$, so there are at least
\[\binom{\fl{m^{1/2}}}{k} \order_k m^{k/2} \geq m^{1/2}\]
cycles in $\binom{\Omega}{k}$. In the latter case, fix a cycle $C$ in $\Omega$ of length at least $m^{1/2}$. There are $\binom{|C|}{k}$ $k$-sets contained in $C$, and each cycle in $\binom{C}{k}$ has length at most $|C|$, so there are at least 
\[
\frac{1}{|C|}  \binom{|C|}{k} \order_k |C|^{k-1} \geq m^{1/2}
\]
cycles in $\binom{C}{k}$.
\end{proof}

\begin{lemma}\label{caseiii}
If $r\geq2$, then every $g\in \cS_m \wr \cS_r$ which is nontrivial in the $\cS_r$ factor has at least $m/r$ cycles in its action on $\{1,\dots,m\}^r$.
\end{lemma}
\begin{proof}
Let $g = (\pi_1,\dots,\pi_r;\sigma)$, where $\sigma\neq 1$. Suppose $(1\cdots s)$ is a cycle of $\sigma$, where $s>1$. Then $g$ acts on $\{1,\dots,m\}^s$, and $g^s$ (the $s$-th power of $g$) acts on $\cS_m\wr\cS_s$ as
\[
	(\pi_1\pi_2\cdots\pi_{s-1}\pi_s,\pi_2\pi_3\cdots\pi_s\pi_1,\dots,\pi_s\pi_1\cdots\pi_{s-2}\pi_{s-1};1).
\]
The coordinates appearing here are conjugate to one another, so they have the same number of cycles of each length $i$, say $c_i$. But if $x_1,\dots,x_s$ are each contained in cycles of length $i$, then $(x_1,\dots,x_s)$ is contained in a cycle of length $i$, so the number of cycles of $g^s$ in $\{1,\dots,m\}^s$ of length $i$ is at least
\[
	\frac{(ic_i)^s}{i} \geq i c_i.
\]
The total number of cycles of $g^s$ in $\{1,\dots,m\}^s$ is thus at least
\[
	\sum_{i=1}^m ic_i = m.
\]
Thus $g$ itself has at least $m/s$ cycles in $\{1,\dots,m\}^s$, and in particular at least $m/s\geq m/r$ cycles in $\{1,\dots,m\}^r$.
\end{proof}

\begin{proof}[Proof of Theorem~\ref{primitive}]
Choose $\pi\in\cS_n$ uniformly at random. By Theorem~\ref{bovey} the probability that $\langle\pi\rangle$ has minimal degree greater than $n^{1-\eps}$ is $O_\eps(n^{-1+2\eps})$, so we may assume that $\langle\pi\rangle$ has minimal degree at most $n^{1-\eps}$. Thus, if $\pi\in G$ and $G$ is primitive, then $G$ also has minimal degree at most $n^{1-\eps}$. Consequently, if $n$ is large enough depending on $\eps$, then $G$ must fall into one of the cases of Corollary~\ref{liebecksaxlcor}. We must rule out cases (ii) and (iii).

Since the number of cycles of a random permutation is approximately Poisson with mean $O(\log n)$, we know that all but at most a proportion $O(n^{-100})$ of $\pi\in\cS_n$ have at most $(\log n)^2$ cycles: see for example~\cite[Lemma~2.2]{EFG2}. Thus, by Lemma~\ref{caseii} we may ignore case (ii). The last case we need to consider is when we can identity $\pi$ with an element $(\pi_1,\dots,\pi_r;\sigma)$ of $\CS_m\wr \CS_r$, acting on $\{1,\dots,m\}^r$. Lemma~\ref{caseiii} then allows us to assume that $\sigma=1$. In this case, though, we find that $\pi$ preserves a system of $m=n^{1/r}$ blocks of size $n^{1-1/r}$, the blocks being the sets $B_a =\{(a,a_2,\dots,a_r): 1\le a_2,\dots,a_r\le m\}$, for $1\le a\le m$. Therefore, Theorem~\ref{imprimitive} implies that the proportion of such $\pi\in\cS_n$ is bounded by
\[
	\sum_{r=2}^{O_\eps(1)} I(n,n^{1/r}) \ll_\eps n^{-1}.\qedhere
\]
\end{proof}


\bibliography{higher-dim}
\bibliographystyle{alpha}

\begin{dajauthors}
\begin{authorinfo}[eber]
  Sean Eberhard\\
  London, UK\\
  eberhard\imagedot{}math\imageat{}gmail\imagedot{}com
\end{authorinfo}
\begin{authorinfo}[ford]
  Kevin Ford\\
  Department of Mathematics\\
  1409 West Green Street\\
  University of Illinois at Urabana--Champaign\\
  Urbana, IL 61801, USA\\
  ford\imageat{}math\imagedot{}uiuc\imagedot{}edu
\end{authorinfo}
\begin{authorinfo}[kouk]
  Dimitris Koukoulopoulos\\
  D\'epartement de math\'ematiques et de statistique\\
  Universit\'e de Montr\'eal\\
  CP 6128 succ. Centre-Ville\\
  Montr\'eal, QC H3C 3J7, Canada\\
  koukoulo\imageat{}dms\imagedot{}umontreal\imagedot{}ca
\end{authorinfo}
\end{dajauthors}
 
\end{document}